\documentclass[11pt]{amsart}
	
\title{Tube algebra of group-type subfactors}

\author[ \large
D\lowercase{ietmar} B\lowercase{isch},
P\lowercase{aramita} D\lowercase{as},
S\lowercase{hamindra} G\lowercase{hosh},
V\lowercase{ed} G\lowercase{upta}, 
N\lowercase{arayan} R\lowercase{akshit},
]
{\bf \large 
	D\lowercase{ietmar} B\lowercase{isch},
	P\lowercase{aramita} D\lowercase{as},
	S\lowercase{hamindra} K\lowercase{umar} G\lowercase{hosh and}
	N\lowercase{arayan} R\lowercase{akshit}}

\thanks{}

\date{}
\address{Department of Mathematics, Vanderbilt University, Nashville, USA}
\email{dietmar.bisch@vanderbilt.edu}
\address{Stat-Math Unit, Indian Statistical Institute, Kolkata, INDIA}
\email{paramita.das@isical.ac.in, shami@isical.ac.in, narayan753@gmail.com}

\usepackage{bbm	}
\usepackage{mathtools}
\usepackage{cancel}
\usepackage{amsmath}
\usepackage{amsfonts}
\usepackage{latexsym}
\usepackage{amssymb}
\usepackage{mathrsfs}
\usepackage{amscd}
\usepackage{hyperref}
\usepackage{soul}
\usepackage{psfrag}
\usepackage{graphicx}
\usepackage{ulem}
\usepackage{fullpage}
\usepackage[all]{xy}
\usepackage{rotating}
\usepackage{url}
\usepackage{color}
\usepackage{enumerate}
\usepackage{cleveref}

\numberwithin{equation}{section}
\numberwithin{figure}{section}
\theoremstyle{plain}
\newtheorem{thm}{Theorem}[section]
\theoremstyle{plain}
\newtheorem{lem}[thm]{Lemma}
\theoremstyle{remark}
\newtheorem{rem}[thm]{Remark}
\theoremstyle{plain}
\newtheorem{cor}[thm]{Corollary}
\theoremstyle{definition}

\theoremstyle{plain}
\newtheorem{prop}[thm]{Proposition}
\theoremstyle{plain}

\newcommand{\comments}[1]{}
\newcommand{\ra}{\rightarrow}
\newcommand{\rab}{\rangle}
\newcommand{\lra}{\longrightarrow}

\newcommand{\lab}{\langle}

\newcommand{\mcal}{\mathcal}

\newcommand{\N}{\mathbb N}

\newcommand{\C}{\mathbb{C}}

\newcommand{\mscr}{\mathscr}
\newcommand{\vlon}{\varepsilon}

\newcommand{\vphi}{\varphi}


\begin{document}
\global\long\def\vlon{\varepsilon}
\global\long\def\bt{\bowtie}
\global\long\def\ul#1{\underline{#1}}
\global\long\def\ol#1{\overline{#1}}
\global\long\def\abs#1{\left|{#1}\right	|}
\global\long\def\norm#1{\left\|{#1}\right\|}
\global\long\def\os#1#2{\overset{#1}{#2}}
\global\long\def\us#1#2{\underset{#1}{#2}}
\global\long\def\ous#1#2#3{\overset{#1}{\underset{#3}{#2}}}
\global\long\def\t#1{\text{#1}}
\global\long\def\lrsuf#1#2#3{\vphantom{#2}_{#1}^{\vphantom{#3}}#2^{#3}}
\global\long\def\tr{\triangleright}
\global\long\def\tl{\triangleleft}
\global\long\def\cc90#1{\begin{sideways}#1\end{sideways}}
\global\long\def\turnne#1{\begin{turn}{45}{#1}\end{turn}}
\global\long\def\turnnw#1{\begin{turn}{135}{#1}\end{turn}}
\global\long\def\turnse#1{\begin{turn}{-45}{#1}\end{turn}}
\global\long\def\turnsw#1{\begin{turn}{-135}{#1}\end{turn}}
\global\long\def\fusion#1#2#3{#1 \os{\textstyle{#2}}{\otimes} #3}

\global\long\def\red#1{\textcolor{red}{#1}}

\global\long\def\1disc#1#2#3#4{\;\;
	{\psfrag{1}{\reflectbox{#1}}
	\psfrag{2}{#2}
	\psfrag{3}{#3}
	\psfrag{4}{#4}
	\includegraphics[scale=0.2]{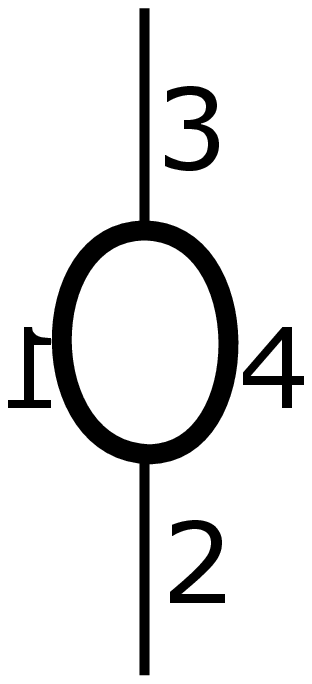}}
	\;\;}

\global\long\def\tower#1#2#3#4{\;\;
	{\psfrag{1}{\reflectbox{#1}}
		\psfrag{2}{#2}
		\psfrag{3}{#3}
		\psfrag{4}{#4}
		\includegraphics[scale=0.2]{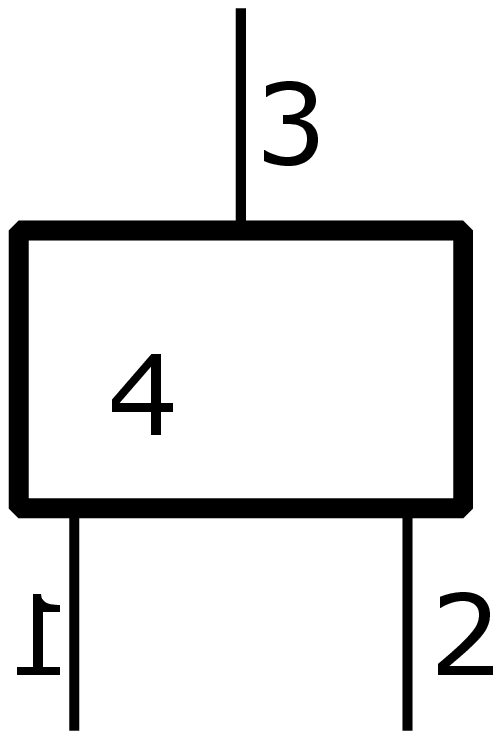}}
	\;\;}

\global\long\def\wine#1#2#3#4{\;\;
	{\psfrag{1}{\reflectbox{#1}}
		\psfrag{2}{#2}
		\psfrag{3}{#3}
		\psfrag{4}{#4}
		\includegraphics[scale=0.2]{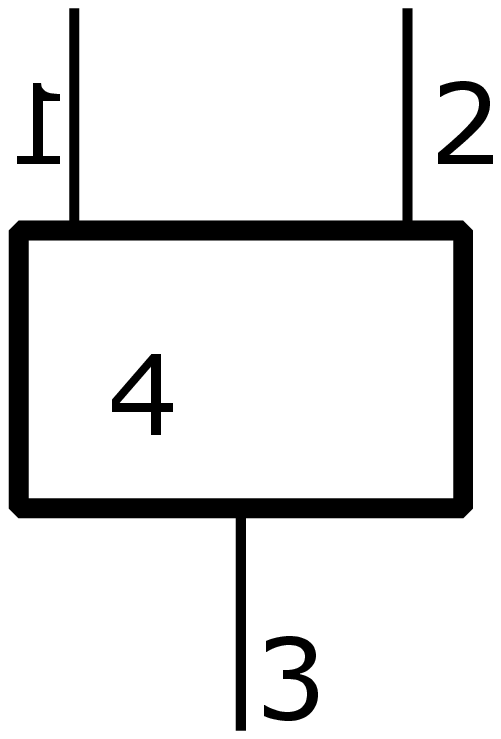}}
	\;\;}

\global\long\def\hexagon#1#2#3#4#5#6#7#8#9{
\psfrag{1}{\reflectbox{#1}}
\psfrag{2}{#2}
\psfrag{3}{#3}
\psfrag{4}{\reflectbox{#4}}
\psfrag{5}{#5}
\psfrag{6}{#6}
\psfrag{7}{\reflectbox{#7}}
\psfrag{8}{#8}
#9
\includegraphics[scale=0.2]{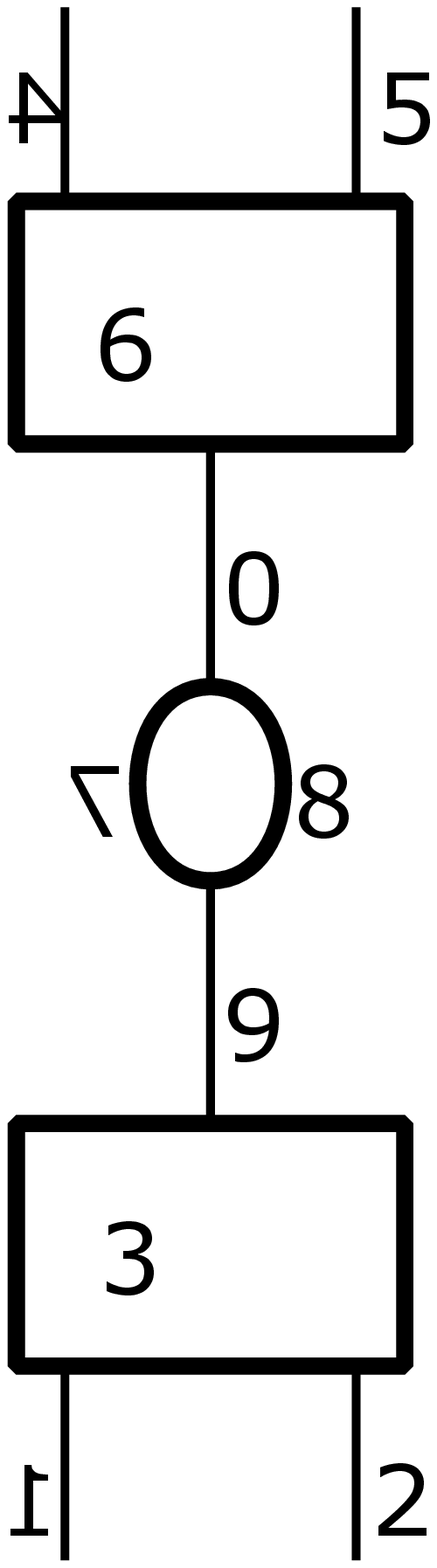}
}

\global\long\def\smbox#1#2#3#4#5#6#7{#1 \raisebox{#7 em}{$\ous {\displaystyle #3} { \scalebox{#5}[#6]{$\square$} } {\displaystyle #2} $}  #4}

\global\long\def\tridown#1#2#3#4#5#6{\raisebox{#6 em}{$\os {\displaystyle #3} {#1 \negthickspace \negthickspace   {\scalebox{#4}[#5]{$ \triangledown $ }} \negthickspace \negthickspace \negthickspace \negthickspace \negthickspace #2}$}}

\global\long\def\triup#1#2#3#4#5#6{{}^{\displaystyle#1} \negthickspace \negthickspace \negthickspace  { \raisebox{#6 em}{ $ \us {\displaystyle #3} { \scalebox{#4}[-#5]{$ \triangledown $ } } $ } }  \negthickspace \negthickspace \negthickspace \negthickspace \negthickspace \negthickspace {}^{\displaystyle#2} }

\maketitle

\begin{abstract}
We describe the tube algebra and its representations in the cases of diagonal and Bisch-Haagerup subfactors possibly with a scalar $ 3 $-cocycle obstruction. We show that these categories are additively equivalent to the direct product over conjugacy classes of representation category of a centralizer subgroup (corresponding to the conjugacy class) twisted by a scalar $ 2 $-cocycle obtained from the $ 3 $-cocycle obstruction.
\end{abstract}
\section{Introduction}
Annular representations of planar algebras were introduced by Vaughan Jones in \cite{Jon3} to construct subfactors with principal graphs $E_6$ and $E_8$. In the same paper,
he explicitly worked out the Temperley-Lieb example.
These calculations helped in construction of new examples such as \cite{Pet}.
Recently, annular representations of subfactors and semisimple rigid $ C^* $-tensor categories have become a very interesting area of research.
The annular representation category turns out to be a nice braided tensor category - not necessarily semisimple - which is equivalent to the center of the original bimodule / $ C^* $-tensor category in the case of finite depth / fusion categories (see \cite{DGG1}, \cite{DGG2}, \cite{GJ}).
For general depth, this category becomes equivalent to the center of a certain induced category (which is basically an extension where infinite direct sums are allowed) - see \cite{NY}, \cite{PV}.
There is also an analytic aspect of the annular representation category.
Analytic properties, such as, amenabilty, Haagerup property, property (T) of the subfactor / $ C^* $-tensor category can be reinterpreted in terms of annular representations.

In this paper, we deal with two group-type subfactors - the so-called diagonal and the Bisch-Haagerup ones - possibly with $ 3 $-cocyle obstruction.
The approximation properties of these two examples are well-known and depend on the associated group (see \cite{frenchpop}, \cite{Pop94} for diagonal and \cite{BH} for Bisch-Haagerup).
We determine the annular representation category.
For this, we borrow techniques from \cite{GJ}, namely, we fix a `full' weight set in the object space and find the annular category over the weight set.
It was shown in \cite{GJ} that the annular representation categories  over any two full weight sets are equivalent.
Moreover, annular representations (in the sense of Vaughan Jones) of a subfactor $ N\subset M $  are the same as the annular representations of the $ N$-$N $-bimodule category $ \mcal C_{NN} $ generated by $ {}_N L^2 (M) {}_N $.

For the diagonal subfactor, $ \mcal C_{NN} $ is a pointed category, that is, the category of group graded vector spaces with a possibly nontrivial associator; the group is the one generated by the automorphisms used to build the diagonal subfactor and the associator is given by the $ 3 $-cocycle obstruction.
When the cocycle is trivial, the annular representations were discussed in \cite{GJ}.
We consider the annular algebra over the irreducible bimodules, that is, Ocneanu's tube algebra.
We show that the tube algebra is a direct sum over conjugacy classes of $ * $-algebras consisting of a matrix algebra tensored with the group algebra of the centralizer subgroup twisted by a $ 2 $-cocycle.
We give the explicit dependence of the $ 2 $-cocycle on the $ 3 $-cocycle obstruction. As a result, the annular representation decomposes as (possibly infinite) direct sum of projective representations of the centralizer subgroups corresponding to the conjugacy classes.

In the Bisch-Haagerup case, we consider the intermediate subfactor $ N \subset Q \subset M $ where $ N= Q^H $ and $ M = Q \rtimes K $ with $H$, $K$ being finite groups acting outerly on the $II_1$ factor $Q$.
The category of $ Q$-$Q $-bimodules $ \mcal C_{QQ} $ generated by $ {}_Q L^2(Q) \us N \otimes L^2(Q) {}_Q  $ and $ {}_Q L^2(M) {}_Q  $, is again a pointed category equivalent to the category of $ G $-graded vector spaces where $ G $ is the group generated by $ H $ and $ K $ in $ \t{Out} (Q) $ with the associator given by the $ 3 $-cocycle obstruction.
This category has special algebra objects, namely, $ A = \us {h \in H} \bigoplus {}_Q L^2(Q_h) {}_Q $ and $ B = \us {k \in K} \bigoplus {}_Q L^2(Q_k) {}_Q $.
Now, the $ 2 $-category of bimodules over $ N $ and $ M $ can be made equivalent to the $ 2 $-category of bimodules in $ \mcal C_{QQ} $ over $ A $ and $ B $.
This was a method suggested to us by Scott Morrison.
However, we obtain the annular representations straight from the actual bifinite $ N$-$N $-bimodules using techniques in \cite{GJ}.
$ \mcal C_{NN} $ unfortunately is not pointed anymore.
The set of isomorphism classes of irreducibles turns out to be complicated.
So, we consider a different weight set, namely $ {}_N L^2(Q_g) {}_N  $ for $ g\in G $ and obtain the annular algebra over it.
As a $ * $-algebra, it turns out to be same as before, namely direct sum over conjugacy classes of matrix algebras tensored with the group algebra of the centralizer subgroup twisted by a $ 2 $-cocycle. Thereby, the representations are also graded by the conjugacy classes and in each grade the representations are the same as that of the corresponding centralizer subgroup twisted by the $ 2 $-cocyle.

\noindent\textbf{Acknowledgement.} The authors would like to thank Scott Morrison and Corey Jones for several useful discussions. A part of this work was completed during the trimester program on von Neumann algebras at the Hausdorff Institute of Mathematics and the authors would like to thank  HIM for the opportunity. The first named author was supported by US NSF grants DMS-0653717 and DMS-1001560, and the Simons Foundation Collaboration Grant no. 359625. 

\section{Some basics on group cocycles}

Let $G$ be a group with identity element $e$ and $\omega \in Z^3(G, S^1)$ be a 3-cocycle of $G$, that is, $\omega$ satisfies the following:
\begin{equation}\label{3coc}
\omega(g_1,g_2,g_3) \omega(g_1,g_2g_3,g_4)\omega(g_2,g_3,g_4) = \omega(g_1g_2,g_3,g_4)\omega(g_1,g_2,g_3g_4) \text{ for all } g_1,g_2,g_3,g_4 \in G
\end{equation}
We will use Equation \ref{3coc} at various instances in the article by denoting the particular elements of $G$ which will correspond to $g_1,g_2,g_3,g_4$ simply by $1,2,3,4$ respectively.
Up to $3$-coboundary equivalence, we may consider  $\omega$ to be a normalized cocycle, i.e.,  $\omega(g_1, g_2, g_3) = 1$ whenever either $g_1, g_2$ or $g_3$ is $e$; namely, if $\varphi(g_1,g_2) = \omega(g_1, e, e) \overline{\omega}(e, e, g_2)$ for all $g_1, g_1 \in G$, then $(\partial \varphi) \omega$ is normalized.

For $a \in G$, let $G_a$  denote the centralizer subgroup of $a$.
The following result may be well-known to specialists but we include the statement for the sake of completeness.  
\begin{lem}\label{2coc}
	$G_a \times G_a \ni (g,h) \stackrel{\varphi_a}{\longmapsto} \overline{\omega}(a,g,h)\omega(g,a,h)\overline{\omega}(g,h,a) $ is a $2$-cocycle of $G_a$.
\end{lem}
\begin{proof}
Note that $\varphi_a (h_2,h_3)\ol{ \varphi}_a (h_1h_2,h_3) \varphi_a (h_1,h_2h_3)\ol{ \varphi}_a (h_1,h_2)$ contain twelve terms involving $\omega$.

The product of the four $\omega$-terms with $a$ in the first place is
\[
{\omega} (\us{1}{a},\us{2}{h_1},\us{3}{h_2}) {\omega} (\us{1}{a},\us{2~3}{h_1h_2},\us{4}{h_3}) \ol{ \omega} (\us{1}{a},\us{2}{h_1},\us{3~4}{h_2h_3}) \ol{\omega} (a,h_2,h_3) = {\omega} (ah_1,h_2,h_3) \ol{\omega} (h_1,h_2,h_3) \ol{\omega} (a,h_2,h_3),
\]
the product of the four $\omega$-terms with $a$ in the third place is
\[
\ol{\omega} (\us{2}{h_2},\us{3}{h_3},\us{4}{a}) \omega(\us{1~2}{h_1h_2}, \us{3}{h_3}, \us{4}{a}) \ol{\omega}(\us{1}{h_1},\us{2~3}{h_2h_3},\us{4}{a}) \omega(h_1,h_2,a) = \ol{\omega}(h_1,h_2,h_3a) \omega(h_1,h_2,h_3)  \omega(h_1,h_2,a),
\]
and the remaining product is
\begin{align*}
& \left[\omega(\us{2}{h_2},\us{3}{a},\us{4}{h_3}) \ol{\omega}(\us{1~2}{h_1h_2},\us{3}{a},\us{4}{h_3})\right] \left[\omega(\us{1}{h_1},\us{2}{a},\us{3~4}{h_2h_3}) \ol{\omega}(\us{1}{h_1},\us{2}{a},\us{3}{h_2})\right]\\
= ~ & \ol{\omega}(h_1,h_2,a)\ol{\omega}(h_1,h_2a,h_3)\omega(h_1,h_2,ah_3) \omega(h_1,ah_2,h_3) \omega(a,h_2,h_3) \ol{\omega}(h_1a,h_2,h_3).
\end{align*}
Now, since $h_1,h_2,h_3$ commute with $a$, all terms in the grand product cancel amongst each other.
\end{proof}
Instead of $a$, if we take $xax^{-1}$ for any $x \in G$, then it is natural to ask whether ${\varphi}_{xax^{-1}} \circ ({\mbox{Ad}}_x \times {\mbox{Ad}}_x) : G_a \times G_a \ra S^1$ is coboundarily equivalent to $\varphi_a$. The
answer is yes; however, we will prove not just this, but a slightly general formula which will be useful later.
\begin{prop}\label{go}
	For all $a,x,y\in G$, there exists ${\gamma}_{a,x,y}: G_a \ra S^1$ such that
	 \begin{align*}
	 & \ol{\omega}(xax^{-1}, xgy^{-1}, yhz^{-1}) \omega(xgy^{-1}, yay^{-1}, yhz^{-1}) \ol{\omega}(xgy^{-1}, yhz^{-1}, zaz^{-1})\\
	  =~ & \left[\gamma_{a,x,y}(g) \gamma_{a,y,z}(h) \ol{\gamma}_{a,x,z}(gh)\right] \varphi_a(g,h)
	 \end{align*}
for all $g,h \in G_a$.
Thus, $\gamma_{a,x,x}$ is a scalar $1$-cochain of $G_a$ which implements the coboundary equivalence between $\varphi_{xax^{-1}}\circ (\t{Ad}_x \times \t{Ad}_x)$ and $\varphi_a$.
\end{prop}
\begin{proof}
We write out the three terms in the L.H.S. of the equation in the statement one by one and expand them using Equation \ref{3coc}. 
In the successive steps, we just repeat the process, each time expanding the last term coming from the previous step. In the final step, some of the terms are decorated with numbers and strike-throughs, or underlines and alphabets, the explanation for which is given below. These are just elementary cocycle calculations which has been exhibited, down to the last detail. 
The first term is:\\
$\t{ }\t{ }\t{ }\ol{\omega}(xax^{-1}, xgy^{-1}, yhz^{-1})$\\
$=\omega(x, ax^{-1}, xghz^{-1}) \ol{\omega}(x, agy^{-1}, yhz^{-1}) \ol{\omega}(x, ax^{-1}, xgy^{-1}) \ol{\omega}(ax^{-1}, xgy^{-1}, yhz^{-1})$\\
$=\omega(x, ax^{-1}, xghz^{-1}) \ol{\omega}(x, agy^{-1}, yhz^{-1}) \ol{\omega}(x, ax^{-1}, xgy^{-1}) \omega(a, x^{-1}, xghz^{-1}) \ol{\omega}(a, x^{-1}, xgy^{-1})$\\
$\t{ }\t{ }\t{ }\ol{\omega}(x^{-1}, xgy^{-1}, yhz^{-1}) \ol{\omega}(a, gy^{-1}, yhz^{-1})$\\
$=\omega(x, ax^{-1}, xghz^{-1}) \ol{\omega}(x, agy^{-1}, yhz^{-1}) \ol{\omega}(x, ax^{-1}, xgy^{-1}) \omega(a, x^{-1}, xghz^{-1}) \ol{\omega}(a, x^{-1}, xgy^{-1})$\\
$\t{ }\t{ }\t{ }\ol{\omega}(x^{-1}, xgy^{-1}, yhz^{-1}) \ol{\omega}(ag, y^{-1}, yhz^{-1}) \omega(a,g,y^{-1}) \omega(g, y^{-1},yhz^{-1}) \ol{\omega}(a,g,hz^{-1})$\\
$=\us{A}{\underbracket{\omega(x, ax^{-1}, xghz^{-1})}} \cancelto{1}{\ol{\omega}(x, agy^{-1}, yhz^{-1})} \us{A}{\underbracket{\ol{\omega}(x, ax^{-1}, xgy^{-1})}} \us{B}{\underbracket{\omega(a, x^{-1}, xghz^{-1})}} \us{B}{\underbracket{\ol{\omega}(a, x^{-1}, xgy^{-1})}}$\\
$\t{ }\t{ }\t{ }\ol{\omega}(x^{-1}, xgy^{-1}, yhz^{-1}) \cancelto{2}{\ol{\omega}(ag, y^{-1}, yhz^{-1})} \us{C}{\underbracket{\omega(a,g,y^{-1})}} \cancelto{7}{\omega(g, y^{-1},yhz^{-1})} \cancelto{3}{\omega(ag, h, z^{-1})}$\\
$\t{ }\t{ }\t{ } \cancelto{6}{\ol{\omega}(g,h,z^{-1})} \us{C}{\underbracket{\ol{\omega}(a, gh,z^{-1})}} \boxed{\ol{\omega}(a,g,h)}$\\
The second term is: \\
$\t{ }\t{ }\t{ }\omega(xgy^{-1}, yay^{-1}, yhz^{-1})$\\
$= \omega(x, gy^{-1}, yay^{-1}) \omega(x, gay^{-1}, yhz^{-1}) \ol{\omega}(x, gy^{-1}, yahz^{-1}) \omega(gy^{-1}, yay^{-1}, yhz^{-1})$\\
$=\omega(x, gy^{-1}, yay^{-1}) \omega(x, gay^{-1}, yhz^{-1}) \ol{\omega}(x, gy^{-1}, yahz^{-1}) \omega(g, y^{-1}, yay^{-1}) \omega(y^{-1}, yay^{-1}, yhz^{-1})$\\
$\t{ }\t{ }\t{ }\ol{\omega}(g, y^{-1}, yahz^{-1}) \omega(g, ay^{-1}, yhz^{-1})$\\
$=\omega(x, gy^{-1}, yay^{-1}) \omega(x, gay^{-1}, yhz^{-1}) \ol{\omega}(x, gy^{-1}, yahz^{-1}) \omega(g, y^{-1}, yay^{-1}) \omega(y^{-1}, yay^{-1}, yhz^{-1})$\\
$\t{ }\t{ }\t{ }\ol{\omega}(g, y^{-1}, yahz^{-1}) \ol{\omega}(g, a, y^{-1}) \ol{\omega}(a, y^{-1}, yhz^{-1}) \omega(ga, y^{-1}, yhz^{-1}) \omega(g, a, hz^{-1})$\\
$=\us{F}{\underbracket{\omega(x, gy^{-1}, yay^{-1})}} \cancelto{1}{\omega(x, gay^{-1}, yhz^{-1})} \cancelto{8}{\ol{\omega}(x, gy^{-1}, yahz^{-1})} \us{E}{\underbracket{\omega(g, y^{-1}, yay^{-1})}} \omega(y^{-1}, yay^{-1}, yhz^{-1})$\\
$\t{ }\t{ }\t{ } \cancelto{5}{\ol{\omega}(g, y^{-1}, yahz^{-1})} \us{D}{\underbracket{\ol{\omega}(g, a, y^{-1})}} \us{B}{\underbracket{\ol{\omega}(a, y^{-1}, yhz^{-1})}} \cancelto{2}{\omega(ga, y^{-1}, yhz^{-1})} \us{C}{\underbracket{\omega(a, h, z^{-1})}} \cancelto{4}{\omega(g, ah, z^{-1})}$\\
$\t{ }\t{ }\t{ } \cancelto{3}{\ol{\omega}(ga, h, z^{-1})} \boxed{\omega(g, a, h)}$\\
The third term is: \\
$\t{ }\t{ }\t{ }\ol{\omega}(xgy^{-1}, yhz^{-1}, zaz^{-1})$\\
$=\omega(x, gy^{-1}, yhaz^{-1}) \ol{\omega}(x, gy^{-1}, yhz^{-1})
\ol{\omega}(x, ghz^{-1}, zaz^{-1}) \ol{\omega}(gy^{-1}, yhz^{-1}, zaz^{-1})$\\
$=\omega(x, gy^{-1}, yhaz^{-1}) \ol{\omega}(x, gy^{-1}, yhz^{-1})
\ol{\omega}(x, ghz^{-1}, zaz^{-1}) \omega(g, y^{-1}, yhaz^{-1})  \ol{\omega}(g, y^{-1}, yhz^{-1})$\\
$\t{ }\t{ }\t{ }\ol{\omega}(y^{-1}, yhz^{-1}, zaz^{-1}) \ol{\omega}(g, hz^{-1}, zaz^{-1})$\\
$=\omega(x, gy^{-1}, yhaz^{-1}) \ol{\omega}(x, gy^{-1}, yhz^{-1})
\ol{\omega}(x, ghz^{-1}, zaz^{-1}) \omega(g, y^{-1}, yhaz^{-1})  \ol{\omega}(g, y^{-1}, yhz^{-1})$\\
$\t{ }\t{ }\t{ }\ol{\omega}(y^{-1}, yhz^{-1}, zaz^{-1}) \omega(g, h, z^{-1}) \omega( h, z^{-1}, zaz^{-1}) 
\ol{\omega}(gh, z^{-1}, zaz^{-1}) \ol{\omega}(g, h, az^{-1})$\\
$=\cancelto{8}{\omega(x, gy^{-1}, yhaz^{-1})} \ol{\omega}(x, gy^{-1}, yhz^{-1})
\us{F}{\underbracket{\ol{\omega}(x, ghz^{-1}, zaz^{-1})}} \cancelto{5}{\omega(g, y^{-1}, yhaz^{-1})}  \cancelto{7}{\ol{\omega}(g, y^{-1}, yhz^{-1})}$\\
$\t{ }\t{ }\t{ }\ol{\omega}(y^{-1}, yhz^{-1}, zaz^{-1}) \cancelto{6}{\omega(g, h, z^{-1})} \us{E}{\underbracket{\omega( h, z^{-1}, zaz^{-1})}} 
\us{E}{\underbracket{\ol{\omega}(gh, z^{-1}, zaz^{-1})}} \us{D}{\underbracket{\omega(gh, a, z^{-1})}}$\\
$\t{ }\t{ }\t{ }\us{D}{\underbracket{\ol{\omega}(h, a, z^{-1})}} \cancelto{4}{\ol{\omega}(g, ha, z^{-1})} \boxed{\ol{\omega}(g, h, a)}$\\
Thus each $\omega$-term has been expressed as a product of $13$
$\omega$-terms. After combining these $39$ terms and noting that \\
(i) $8$ pairs of terms cancel since $g, h \in G_a$ (the cancellations have been marked with numbers for the reader's convenience),\\ 
(ii) the last (boxed) terms on the R.H.S of the three expressions above can be combined to yield $\varphi_a(g, h)$,\\ 
we are left with the following $20$ $\omega$-terms which have been grouped under A,B, C,D, E, F for reasons that will become apparent as we go along (namely, contribution towards defining the function $\gamma$ in the statement of the Proposition.):\\
\noindent
A terms: $= \ol{\omega}(x, ax^{-1}, xgy^{-1}) \omega(x, ax^{-1}, xghz^{-1})$\\
B terms: $\ol{\omega}(a, x^{-1}, xgy^{-1}) \ol{\omega}(a, y^{-1}, yhz^{-1}) \omega(a, x^{-1}, xghz^{-1})$\\
C terms:  $\omega(a,g,y^{-1})  \omega(a, h, z^{-1})  \ol{\omega}(a, gh,z^{-1})$\\
D terms: $\ol{\omega}(g, a, y^{-1}) \ol{\omega}(h, a, z^{-1}) \omega(gh, a, z^{-1}) $\\ 
E terms: $\omega(g, y^{-1}, yay^{-1}) \omega( h, z^{-1}, zaz^{-1}) \ol{\omega}(gh, z^{-1}, zaz^{-1})$\\
F terms: $\omega(x, gy^{-1}, yay^{-1})  \ol{\omega}(x, ghz^{-1}, zaz^{-1})$\\

\noindent The remaining 4 terms are:
\begin{equation}\label{remain}
\ol{\omega}(x^{-1}, xgy^{-1}, yhz^{-1}) \omega(y^{-1}, yay^{-1}, yhz^{-1}) \ol{\omega}(x, gy^{-1}, yhz^{-1}) \ol{\omega}(y^{-1}, yhz^{-1}, zaz^{-1})
\end{equation}

\noindent The second and fourth terms in expression \ref{remain} are again broken up using Equation \ref{3coc} as follows:\\
$\omega(y^{-1}, yay^{-1}, yhz^{-1}) = \ol{\omega}(y^{-1}, y, ay^{-1}) \us{A}{\underbracket{\ol{\omega}(y, ay^{-1}, yhz^{-1})}} \cancelto{9}{\omega(y^{-1}, y, ahz^{-1})}$\\
$\ol{\omega}(y^{-1}, yhz^{-1}, zaz^{-1}) = \us{G}{\underbracket{\omega(y^{-1}, y, hz^{-1})}} \us{F}{\underbracket{\omega(y, hz^{-1}, zaz^{-1})}} \cancelto{9}{\ol{\omega}(y^{-1}, y, haz^{-1})} $\\
and the first and the third terms in \ref{remain} taken together, is:\\
$\ol{\omega}(x^{-1}, xgy^{-1}, yhz^{-1}) \ol{\omega}(x, gy^{-1}, yhz^{-1}) = \us{G}{\underbracket{\omega(x^{-1}, x, gy^{-1})}} \us{G}{\underbracket{\ol{\omega}(x^{-1}, x, ghz^{-1})}}$

We now expand each of the terms in E, using Equation \ref{3coc} again:\\
\begin{tabular}{llll}
& $\omega(g, y^{-1}, yay^{-1})$ & $\omega(h, z^{-1}, zaz^{-1})$ & $\ol{\omega}(gh, z^{-1}, zaz^{-1})$\\
= & $\ol{\omega}(gy^{-1}, y, ay^{-1})$ &$\omega(g, y^{-1}, y)$ & $\omega(y^{-1}, y, ay^{-1})$\\
&$\ol{\omega}(hz^{-1}, z, az^{-1})$ & $\omega(h, z^{-1}, z)$ & $\cancelto{10}{\omega(z^{-1}, z, az^{-1})}$ \\
&  $\omega(ghz^{-1}, z, az^{-1})$ & $\ol{\omega}(gh, z^{-1}, z)$ & $\cancelto{10}{\ol{\omega}(z^{-1}, z, az^{-1})}$ \\
\end{tabular}\\
Call the terms in the first column as $E_1$ and the terms in the second column as $E_2$. The new A and F terms that popped up from breaking down \ref{remain}, and the $E_1$ and $E_2$ terms are added to the existing list:\\
\noindent
A terms: $= \ol{\omega}(x, ax^{-1}, xgy^{-1}) \ol{\omega}(y, ay^{-1}, yhz^{-1}) \omega(x, ax^{-1}, xghz^{-1})$\\
B terms: $\ol{\omega}(a, x^{-1}, xgy^{-1}) \ol{\omega}(a, y^{-1}, yhz^{-1}) \omega(a, x^{-1}, xghz^{-1})$\\
C terms:  $\omega(a,g,y^{-1})  \omega(a, h, z^{-1})  \ol{\omega}(a, gh,z^{-1})$\\
D terms: $\ol{\omega}(g, a, y^{-1}) \ol{\omega}(h, a, z^{-1}) \omega(gh, a, z^{-1}) $\\ 
$E_1$ terms: $\ol{\omega}(gy^{-1}, y, ay^{-1}) \ol{\omega}(hz^{-1}, z, az^{-1}) \omega(ghz^{-1}, z, az^{-1})$\\
$E_2$ terms: $\omega(g, y^{-1}, y) \omega(h, z^{-1}, z) \ol{\omega}(gh, z^{-1}, z)$\\
F terms: $\omega(x, gy^{-1}, yay^{-1})  \omega(y, hz^{-1}, zaz^{-1}) \ol{\omega}(x, ghz^{-1}, zaz^{-1})$\\
G terms: $\omega(x^{-1}, x, gy^{-1}) \omega(y^{-1}, y, hz^{-1}) \ol{\omega}(x^{-1}, x, ghz^{-1})$\\

 Thus we define ${\gamma}_{a,x,y}(g) = \ol{\omega}(x, ax^{-1}, xgy^{-1}) \ol{\omega}(a, x^{-1}, xgy^{-1}) \omega(a,g,y^{-1}) \ol{\omega}(g, a, y^{-1})$\\
 $\ol{\omega}(gy^{-1}, y, ay^{-1}) \omega(g, y^{-1}, y) \omega(x, gy^{-1}, yay^{-1}) \omega(x^{-1}, x, gy^{-1})$. \\
 This fits the bill.
\end{proof}
\section{Diagonal Subfactors}\label{diag}

In this section, we will describe the affine module category of the planar algebra of a diagonal subfactor associated associated to a `$ G $-kernel' where $ G $ is a finitely generated discrete group.
Recall that $ G $\textit{-kernel} is simply an injective homomorphism $ \chi : G \ra \t {Out} (N) $ where $ N $ is a $ II_1 $ factor.
If $ \alpha : G \ra \t {Aut} (N) $ is a lift of $ \chi $ (that is, $ \chi (g) = \alpha_g \; \t {Inn} (N)$ for all $ g \in G $) and $ I $ is a set of generators of $ G $, then the associated diagonal subfactor is given by
\[
N \ni x \hookrightarrow \t{diag} (\alpha_i (x))_{i\in I} \in M_I (N) =: M.
\]
Further for $ g_1,g_2 \in G$, we may choose $ u(g_1,g_2) \in \mcal U (N) $ such that $ \alpha_{g_1} \alpha_{g_2} = \t {Ad}_{u(g_1,g_2)} \alpha_{g_1g_2} $ for all $ g_1 , g_2 \in G $.
Associativity of multiplication in $ G $ gives us a $ 3 $-cocycle $ \omega : G^{ \times 3} \ra S^1 $ such that
\begin{equation}\label{uomega}
u(g_1, g_2) \; u(g_1 g_2, g_3) = \omega (g_1, g_2, g_3) \;  \alpha_{g_1} (u(g_2, g_3)) \;  u(g_1, g_2 g_3)
\end{equation}
for $ g_1, g_2, g_3 \in G $.
One may easily  check that the coboundary class of the $ 3 $-cocyle $\omega  $ in $ H^3 (G, S^1) $ does not depend on the choice of the lift $ \alpha $ and the unitaries $ u(\cdot , \cdot) $; this class is referred as the obstruction of the $ G $-kernel $ \chi $.
It is well-known (see \cite{frenchpop}) that the standard invariant of the above subfactor $ N \subset M $ depends only on the group $G$, its generators and the $3$-cocycle obstruction.

We will find the tube algebra of the category $ \mcal C_{NN} $ of $ N $-$ N $ bifinite bimodules coming from this subfactor and then find the tube representations.
Note that this will suffice since, by \cite{GJ}, the representation category of the tube algebra of $ \mcal C_{NN} $ is (tensor) equivalent to the category of annular representations with respect to any full weight set in ob$ (\mcal C_{NN}) $ (in particular, $ \left\{\left(_N L^2(M)_N \right)^{\us N \otimes k} : k\in \N\right\} $ which gives the affine modules of Jones).
In fact, if $\omega \equiv 1$, the affine modules were obtained in \cite{GJ}.

All simple objects in $ \mcal C_{N,N} $ are invertible.
This is clear because $ L^2(M) \us {N\t{-}N } \cong  \us {i\in I} \bigoplus L^2( N_{\alpha_i} ) $. Here the notation $_N L^2 (N_\theta)_N $ (for $ \theta \in \t{Aut} (N) $) denotes the bimodule obtained from the Hilbert space $ L^2 (N) $ with left $ N $-action being the usual left multiplication whereas the right one is twisted by $ \theta $. This bimodule depends only on the class defined by $ \theta $ in $ \t{Out} (N)$ up to isomorphism, and the tensor $ \us N \otimes $ and the contragradient of such bimodules correspond to multiplication and inverse in $ \t{Out} (N)$.
Since $ \t{id}_N $ is in the set $\{\alpha_i : i\in I\}  $, we get all such index one bimodules corresponding to any $ g \in G $, appearing as sub-bimodules of $ \left(_N L^2(M)_N \right)^{\us N \otimes k}  $ as we vary $ k $.
Moreover, up to isomorphism these are the only irreducible bimodules of $\mcal C _{N N} $.
Thus, the fusion algebra of $ \mcal C_{NN} $ is just given by $ G $.
It is then easy to verify that $ \mcal C_{NN} $ is tensor equivalent to the category $ \t {Vec} (G, \omega) $ of $ G $-graded vector spaces with associativity constraint given by the $ 3 $-cocycle obstruction $ \omega $.
So, our job boils down to finding out the tube representations of $ \t {Vec} (G, \omega) $.
However, we will work with bimodules in $ \mcal C_{NN} $ instead, as the framework will be useful in the next section.

Since the standard invariant (and thereby the category $ \mcal C_{NN} $) is independent of the lift $ \alpha $, without loss of generality we assume
$ \alpha_e = \t {id}_N $.
Further, we may set $ u(g_1,e) = 1_N = u(e,g_2) $ for all $ g_1 , g_2 \in G $.
These assumptions make the $ 3 $-cocycle $ \omega : G^{ \times 3} \ra S^1 $ normalized.

For $ g \in G $, let $ X_g := _N L^2 (N_{\alpha_g})_N$.
The morphism space in $ \mcal C_{NN} $ from object $ U $ to object $ V $, will be denoted by $ \mcal C_{NN} (U,V) $.
The tube morphism from $ X_{g_1} $ to $ X_{g_2} $ is then given by $ \mcal T_{g_1, g_2} := \us {s\in G} \bigoplus \mcal T^s_{g_1,g_2}$ where $\mcal T^s_{g_1,g_2} = \mcal C_{NN} (X_{g_1} \otimes X_s , X_s \otimes X_{g_2}  )$.
Clearly, $ \mcal T_{g_1,g_2}  \neq \{0\}$ if and only if $ g_1 $ and $ g_2 $ are conjugates of each other.
Further, if $ g_1 = s g_2 s^{-1} $, then $ \mcal T^s_{g_1, g_2} $ is one-dimensional; we will fix a distinguished element in this space, namely, $ a(g_1 , s , g_2) $ defined by
\[
X_{g_1} \us N \otimes X_s \; \ni \; [1]_{g_1} \us N \otimes [1]_s \; \os {a(g_1,s,g_2)} \longmapsto \; [u(g_1 , s) u^*(s, g_2)]_s \us N \otimes [1]_{g_2} \; \in \; X_s \us N \otimes X_{g_2}.
\]
It is an easy exercise to check that the above map is indeed an $ N $-$ N $-linear unitary.

Before we multiply two nonzero tube morphisms $ a(g_1 , s , g_2) $ and $ a(g_2, t , g_3) $, we need to know the one dimensional spaces $ \mcal C_{NN} (X_s \us N \otimes X_t , X_{st})  = \C \left\{[1]_s \us N \otimes [1]_t \os {\beta_{s,t}} \longmapsto [u(s,t) ]_{st}\right\}$ and $ \mcal C_{NN} (X_{st} , X_s \us N \otimes X_t) = \C \left\{ [1]_{st} \os {\beta^*_{s,t}} \longmapsto  [u^* (s,t)]_s \us N \otimes [1]_t \}\right\}$.
Following the multiplication defined in \cite[Section 3]{GJ}, we have
\[
a(g_2,t,g_3) \cdot a(g_1,s,g_2) = \beta_{s,t} \us N \otimes \t {id}_{X_{g_3}} \; \circ \; \t {id}_{X_s} \us N \otimes a(g_2 ,t , g_3) \; \circ \; a(g_1, s, g_2) \us N \otimes \t {id}_{X_t} \; \circ \; \t {id}_{X_{g_1}} \us N \otimes \beta^*_{s,t} .
\]
Right from the definitions, one can easily see that $ a(g_2,t,g_3) \cdot a(g_1,s,g_2) $ sends $ [1]_{g_1} \us N \otimes [1]_{st} $ to
\[
\left[ \alpha_{g_1} (u^* (s,t)) \; u(g_1,s) \; u^* (s,g_2) \; \alpha_s \left( u(g_2, t) u^* (t,g_3) \right) \; u(s,t) \right]_{st} \us N \otimes [1]_{g_3}
\]
Now,
\begin{align*}
&\; \alpha_{g_1} (u^* (s,t)) \; u(g_1,s) \; u^* (s,g_2) \; \alpha_s \left( u(g_2, t) u^* (t,g_3) \right) \; u(s,t)\\
= &\;  \omega (g_1,s,t) \; u(g_1,st) \; u^* (\us {=s g_2} {g_1 s} ,t) \; u^* (s,g_2) \; \alpha_s ( u(g_2, t) )  \; \alpha_s ( u^* (t,g_3) ) \; u(s,t)\\
&\; \t {(using Equation \ref{uomega} on the first two terms)}\\
= &\;  \omega (g_1,s,t) \; u(g_1,st) \; \ol \omega (s , g_2 , t) \; u^* (s, \us {= t g_3} {g_2 t})  \; \alpha_s ( u^* (t,g_3) ) \; u(s,t)\\
&\; \t {(using Equation \ref{uomega} on the third, fourth and fifth terms)}\\
= &\;  \omega (g_1,s,t) \; u(g_1,st) \; \ol \omega (s , g_2 , t)  \; \omega (s,t,g_3) \; u^*(st,g_3)\\
&\; \t {(using Equation \ref{uomega} on the last three terms)}\\
= &\; \left[\omega (g_1,s,t) \; \ol \omega (s , g_2 , t)  \; \omega (s,t,g_3)\right] \; u(g_1,st) \; u^*(st,g_3)
\end{align*}
Thus, multiplication is given by
\[
a(g_2,t,g_3) \cdot a(g_1,s,g_2) =  \left[\omega (g_1,s,t) \; \ol \omega (s , g_2 , t)  \; \omega (s,t,g_3)\right] a(g_1 , st , g_3).
\]

Next we will obtain the $ * $-structure on the tube algebra which we denote by $ \# $ following the notation in \cite{GJ}.
For this, we need a standard solution to the conjugate equations for the pair $ (X_s , X_{s^{-1}}) $.
We set $ R_s := \beta^*_{s^{-1},s} : X_e \ra X_{s^{-1}} \us N \otimes X_s$ and $ \ol R_s := \left[ \omega (s, s^{-1} , s) \beta^*_{s, s^{-1}} \right] : X_e \ra X_s \us N \otimes X_{s^{-1}} $.
It is completely routine to check that $ (R_s , \ol R_s) $ satisfies the conjugate equation and is standard.
Now by \cite{GJ},
\[
\left( a(g_1, s , g_2)\right)^\# = \left[\t {id}_{X_{s^{-1}} } \us N \otimes \t {id}_{X_{g_1}} \us N \otimes (\ol R_s)^*\right] \circ \left[\t {id}_{X_{s^{-1}} } \us N \otimes (a(g_1,s,g_2))^* \us N \otimes \t {id}_{X_{s^{-1}} }\right] \circ \left[R_s \us N \otimes \t {id}_{X_{g_2} } \us N \otimes \t {id}_{X_{s^{-1}} }\right].
\]
The map $ (a(g_1 , s , g_2))^* $ sends $ [1]_s \us N \otimes [1]_{g_2} $ to $ [u(s,g_2) u^* (g_1,s)]_{g_1} \us N \otimes [1]_s $.
Using all the three maps $ (a(g_1 , s , g_2))^* $, $ R_s $ and $ \ol R_s $, we can express the image of $ [1]_{g_2} \us N \otimes [1]_{s^{-1}} $ under $ \left( a(g_1, s , g_2)\right)^\# $ as
\[
\left[ \us {= \omega (s^{-1} , s , s^{-1})} {\ol \omega (s , s^{-1} , s)} \; u^*(s^{-1} , s) \; \alpha_{s^{-1}} \left( u(s,g_2) u^* (g_1 , s) \right) \; \alpha_{s^{-1}} (\alpha_{g_1} (u(s,s^{-1}))) \right]_{s^{-1}} \us N \otimes [1]_{g_1}.
\]
We will simplify the first tensor component in the following way:
\begin{align*}
& \; \omega (s^{-1} , s , s^{-1}) \; u^*(s^{-1} , s) \; \alpha_{s^{-1}} ( u(s,g_2) ) \; \alpha_{s^{-1}} \left( u^* (g_1 , s) \alpha_{g_1} (u(s,s^{-1}))\right)\\
= & \; \omega (s^{-1} , s , s^{-1}) \; \ol \omega (s^{-1},s,g_2) \; u^* (s^{-1}, sg_2) \; \ol \omega (g_1 , s , s^{-1}) \; \alpha_{s^{-1}} (u(\us {=sg_2} {g_1s} , s^{-1}))\\
& \; \t{(using Equation \ref{uomega} on the second and third, and fourth and fifth terms separately)}\\
= & \; \omega (s^{-1} , s , s^{-1}) \; \ol \omega (s^{-1},s,g_2) \; \ol \omega (g_1 , s , s^{-1}) \; \ol \omega (s^{-1} , s g_2 , s^{-1}) \; u(g_2 , s^{-1}) \; u^* (s^{-1} , g_1)\\
& \; \t{(using Equation \ref{uomega} on the third and fifth terms)}\\
= & \; \omega (s^{-1} , s , s^{-1}) \; \ol \omega (g_1 , s , s^{-1}) \; \omega (s, g_2, s^{-1}) \; \ol \omega (s^{-1},s, \us {=s^{-1} g_1} {g_2 s^{-1}}) \; u(g_2 , s^{-1}) \; u^* (s^{-1} , g_1)\\
& \; \t{(using Equation \ref{3coc} on the second and fourth terms)}\\
= & \; \left(\ol \omega (g_1 , s , s^{-1}) \; \omega (s, g_2, s^{-1}) \; \ol \omega (s , s^{-1}, g_1)\right) \; \left(u(g_2 , s^{-1}) \; u^* (s^{-1} , g_1)\right)\\
& \; \t{(using Equation \ref{3coc} on the first and fourth terms)}
\end{align*}
Hence, $ \# $ is given by the formula:
\[
\left(a(g_1 , s , g_2)\right)^\# = \left[\ol \omega (g_1 , s , s^{-1}) \; \omega (s, g_2, s^{-1}) \; \ol \omega (s , s^{-1}, g_1)\right] \; a(g_2, s^{-1} , g_1).
\]

The canonical (faithful) trace $ \Omega $ on the tube algebra (as defined in \cite{GJ}) is given by $ \Omega (a(g_1,s,g_2)) = \delta_{g_1 = g_2} \delta_{s=e}$.
Thus, the set $ \{a(g_1,s,g_2) : g_1,g_2,s \in G \t{ satisfying } g_1 s = s g_2 \}$ becomes an orthonormal basis with respect to the inner product arising from $ \Omega $ and $ \# $.

To have a better understanding of the $ * $-algebra structure of the tube algebra, we will now set up some notations.
Let $ \mscr C $ denote the set of conjugacy classes of $ G $.
For each $ C \in \mscr C $, we pick a representative $ g_C \in C $ and for each $ g \in C $, we fix $ w_g \in G $ such that $ g = w_g \; g_C \; w^{-1}_g $ and $ w_{g_C} = e$.
Also for $ C \in \mscr C $, we will denote the centralizer subgroup of $ g_C $ by $ G_C := \{ s \in G : g_C = s \; g_C \; s^{-1}\} $, and $ \vphi_C$ will denote the $ 2 $-cocycle on $ G_C $ given by $\vphi_C (s,t) := \ol \vphi_{g_C} (t^{-1} , s^{-1} ) $ (recall the definition of $ \vphi_{g_C} $ in Lemma \ref{2coc}).

With the above notation, we give an alternate description of $ * $-algebra structure of the tube algebra in the following proposition which will be handy in classifying the representations.
\begin{thm}\label{diagtubalg}
(i) The tube algebra $ \mcal T  = (( \mcal T_{g_1 , g_2} ))_{\t {fin. supp.}}$ is isomorphic to $ \us {C \in \mscr C} \bigoplus \; M_C \otimes \left[\C G_C\right]_{\vphi_C} $ as a $ * $-algebra where $ \left[\C G_C\right]_{\vphi_C}  $ is the $ 2 $-cocycle twisted group algebra and $ M_C $ denotes the $ * $-algebra of finitely supported matrices whose rows and columns are indexed by elements of $ C $.\\
(ii) Every Hilbert space representation $\Pi : \mcal T \ra \mcal L (V)$ decomposes over $ C \in \mscr C $ uniquely (up to isomorphism) as an orthogonal direct sum of submodules generated by the range of the projection $\Pi (a(g_C, e, g_C)) $ (which is the $ {g_C}^{\t {th}} $-space of $ V $).
(We will call a representation of $ \mcal T $ `supported on $ C \in \mscr C$' if it is generated by its vectors in the $ {g_C}^{\t {th}} $-space.)
The category of $ C $-supported representations of $ \mcal T $ is additively equivalent to representation category of $ [\C G_C]_{\vphi_C} $.
\end{thm}
\begin{proof}
(i) We will send the orthonormal basis of $ \mcal T $ (discussed above) via a map $ \Phi $ to a canonical basis of $ \us {C \in \mscr C} \bigoplus \; M_C \otimes \left[\C G_C\right]_{\vphi_C}  $ in the follwoing way:\\
for $ g_1 , g_2 \in C $  and $ s\in G $ such that $ g_1 s = s g_2 $ (implying  $ w^{-1}_{g_1} s w_{g_2} \in G_C $)
\[
a(g_1 , s , g_2) \; \os \Phi \longmapsto \; \ol \gamma_{g_C , w_{g_1} , w_{g_2}} (w^{-1}_{g_1} s w_{g_2}) \; E_{g_2,g_1} \otimes [w^{-1}_{g_2} s^{-1} w_{g_1}] 
\]
where we use the family of functions $\left\{ \gamma_{a,x,y}: G_a \ra S^1  \right\}_{a,x,y\in G}$ appearing in Proposition \ref{go}.

To show $ \Phi $ preserves multiplication, consider
\begin{align*}
& \; \Phi (a(g_2,t,g_3)) \; \Phi (a(g_1,t,g_2))\\
= & \; \left[\ol \gamma_{g_C , w_{g_2} , w_{g_3}} (w^{-1}_{g_2} t w_{g_3}) \; E_{g_3,g_2} \otimes [w^{-1}_{g_3} t^{-1} w_{g_2}]\right] \; \left[\ol \gamma_{g_C , w_{g_1} , w_{g_2}} (w^{-1}_{g_1} s w_{g_2}) \; E_{g_2,g_1} \otimes [w^{-1}_{g_2} s^{-1} w_{g_1}]\right]\\
= & \; \ol \gamma_{g_C , w_{g_1} , w_{g_2}} (w^{-1}_{g_1} s w_{g_2}) \; \ol \gamma_{g_C , w_{g_2} , w_{g_3}} (w^{-1}_{g_2} t w_{g_3}) \; \vphi_C (w^{-1}_{g_3} t^{-1} w_{g_2} \; , \; w^{-1}_{g_2} s^{-1} w_{g_1} ) \; E_{g_3,g_1} \otimes [w^{-1}_{g_3} (st)^{-1} w_{g_1}]\\
= &  \; \ol \gamma_{g_C , w_{g_1} , w_{g_2}} (w^{-1}_{g_1} s w_{g_2}) \; \ol \gamma_{g_C , w_{g_2} , w_{g_3}} (w^{-1}_{g_2} t w_{g_3}) \;  \gamma_{g_C , w_{g_1} , w_{g_3}} (w^{-1}_{g_1} st w_{g_3}) \; \ol \vphi_{g_C} (w^{-1}_{g_1} s w_{g_2} \; , \; w^{-1}_{g_2} t w_{g_3} )\\
& \; \Phi (a(g_1 , st, g_3))\\
= & \; \omega (g_1 ,s , t) \ol \omega (s,g_2 , t) \omega (s , t, g_3) \; \Phi (a(g_1 , st, g_3)) \t { (using Proposition \ref{remain})}\\
= & \; \Phi \left(a(g_2 , t, g_3) \; a(g_1 , s, g_2)\right).
\end{align*}
The map $ \Phi $ is $ * $ preserving because $ \left[ \Phi ( a(g_1,s,g_2) ) \right]^*
 $
\begin{align*}
= & \; \gamma_{g_C, w_{g_1} , w_{g_2} } (w^{-1}_{g_1} s w_{g_2}) \; \ol \vphi_C ( w^{-1}_{g_1} s w_{g_2} , w^{-1}_{g_2} s^{-1} w_{g_1} ) \; E_{g_1, g_2} \otimes [w^{-1}_{g_1} s w_{g_2}]\\
= & \; \gamma_{g_C, w_{g_1} , w_{g_2} } (w^{-1}_{g_1} s w_{g_2}) \; \vphi_{g_C} ( w^{-1}_{g_1} s w_{g_2} , w^{-1}_{g_2} s^{-1} w_{g_1} ) \; \gamma_{g_C , w_{g_2} , w_{g_1} } ( w^{-1}_{g_2} s^{-1} w_{g_1}) \; \Phi (a(g_2,s^{-1},g_1))\\
= & \; \gamma_{g_C , w_{g_1} , w_{g_1} } (e) \; \ol \omega (g_1 , s , s^{-1}) \; \omega (s , g_2 , s^{-1}) \; \ol \omega (s , s^{-1} , g_1) \; \Phi (a(g_2 , s^{-1} , g_1)) \; \t {(using Proposition \ref{remain})}\\
=& \; \Phi \left( \left[  a(g_1,s,g_2)  \right]^\# \right)
\end{align*}
where we use $ \gamma_{g_C , w_{g_1} , w_{g_1} } (e) =1$ at the very last step which follows directly from the definition of $ \gamma_{a,x,y} $ in the proof of Proposition \ref{remain}.

(ii) The decomposition follows easily from the $ * $-algebra structure described in part (i).

Fix $ C\in \mscr C $. If $ \Pi: \mcal T \ra \mcal L (W) $ is $ C $-supported, then we can define the representation $ \pi : [\C G_C]_{\vphi_C} \ra \mcal L (W_{g_C}) $ defined by $ \pi (s) = \Pi\left(\Phi^{-1} (E_{g_C , g_C} \otimes [s])\right) $.
Conversely, if $ \pi : [\C G_C]_{\vphi_C} \ra \mcal L (U) $ is a representation, then one can consider the unique extension 
\[
\Pi : \mcal T \ra \mcal L (l^2 (C) \otimes U) \; \t { defined by } \; \Pi \left( \Phi^{-1} (E_{g_1 , g_2} \otimes [s] ) \right) := \delta_{g_1 \in C} \; E_{g_1,g_2} \otimes \pi (s) .
\]
\end{proof}
\begin{rem}
Note that the canonical trace $ \Omega $ on $ \mcal T $ corresponds to the direct sum of the canonical traces on $ M_C \otimes [ \C G_C]_{\vphi_C} $.
Also, the $ * $-algebra $ \mcal T_{e,e} $ (by definition) is isomorphic to the fusion algebra which is basically the group algebra $\C G $ without any nontrivial $ 2 $-cocycle twist (since $ \vphi_{e}$ is the constant function $1$ which follows from its definition in Lemma \ref{2coc}).
Thus, the analytic properties (such as, amenability, Haagerup, property (T)) of the bimodule category corresponding to the subfactor $ N \subset M $ corresponds exactly to that of the group $ G $; this fact was obtained by Sorin Popa long time back in \cite{frenchpop} and \cite{Pop94}.
However, the analytic properties in the higher weight spaces (as defined in \cite{GJ}) depend on the corresponding centralizer subgroup.
\end{rem}
\comments{\begin{thm}\label{diagthm}
Every Hilbert representation $ V $ of $ \mcal T $ decomposes uniquely (up to isomorphism) as an orthogonal direct sum of submodules $ V^C := \lab V_{g_C} \rab$ for $ C \in \mscr C $.
(We will call a representation of $ \mcal T $ `supported on $ C \in \mscr C$' if it is generated by its vectors in the $ {g_C}^{\t {th}} $-space.)
The category of $ C $-supported representations of $ \mcal T $ is additively equivalent to representation category of $ [\C G_C]_{\vphi_C} $.
\end{thm}
\begin{proof}
The decomposition follows easily from the $ * $-algebra structure described in Proposition \ref{diagtubalg}.

For the second part, fix $ C\in \mscr C $. If $ \Pi: \mcal T \ra \mcal L (W) $ is $ C $-supported, then we can define the representation $ \pi : [\C G_C]_{\vphi_C} \ra \mcal L (W_{g_C}) $ defined by $ \pi (s) = \Pi\left(\Phi^{-1} (E_{g_C , g_C} \otimes [s])\right) $.
Conversely, if $ \pi : [\C G_C]_{\vphi_C} \ra \mcal L (U) $ is a representation, then one can consider the unique extension 
\[
\Pi : \mcal T \ra \mcal L (l^2 (C) \otimes U) \; \t { defined by } \; \Pi \left( \Phi^{-1} (E_{g_1 , g_2} \otimes [s] ) \right) := \delta_{g_1 \in C} \; E_{g_1,g_2} \otimes \pi (s) .
\]
\end{proof}}
\section{Bisch-Haagerup Subfactors}
In this section, we intend to find the tube algebra of the Bisch-Haagerup subfactor $N:= Q^{H} \subset Q \rtimes K = : M$ where $ H $ and $ K $ act outerly  on the $II_1$-factor $Q$.
It is well known that the planar algebra of $ N \subset M $ depends on the group $ G $ generated by $ H $ and $ K $ in $ \t {Out} (Q) $ and the scalar $ 3 $-cocycle obstruction (up to $ 2 $-coboundary) (see \cite{BH, BDG1}).

We first lay down the strategy to achieve our goal.
Instead of computing the tube algebra of $ \mcal C_{NN} $ directly (unlike the case of diagonal subfactors because the irreducible bimodules of $ \mcal C_{NN} $ for Bisch-Haagerup subfactors, are not so easy to work with), we will consider the affine annular algebra with respect to a particular full weight set (in the sense of  \cite[Definition 3.4]{GJ}) in $\t {ob} (\mcal C_{NN}) $, and then cut it down by the $ \t{Irr} \; \mcal C_{NN} $.

We need to set up some notations for this.
Pick a representative map $ \t {Out} (Q) \supset G \ni g \mapsto \alpha_g \in \t {Aut} (Q) $ such that $ g = \alpha_g \t {Inn} (Q) $, and $ \left.\alpha\right|_H : H \ra \t {Aut} (Q)$, $ \left.\alpha\right|_K : K \ra \t {Aut} (Q) $ are homomorphisms.
Now, if $ X = _N L^2 (M)_M $, $ Y = _N L^2 (Q)_Q $ and  $ Z = _Q L^2 (M)_M $, then
\[
\left( \; X \; \ol X \; \right)^{\us N \otimes m} = Y \us Q \otimes ( \; Z \; \ol Z \; ) \us Q \otimes ( \; \ol Y \; Y \; ) \us Q \otimes ( \; Z \; \ol Z \; ) \us Q \otimes \cdots \us Q \otimes ( \; Z \; \ol Z \; ) \us Q \otimes \ol Y.
\]
We know that $ ( \; \ol Y \; Y \; ) \us {Q\t {-} Q }\cong \; \us {h\in H} \bigoplus \; _Q  L^2 (Q_{\alpha_h})_Q$ and $ ( \; Z \; \ol Z \; ) \us {Q\t {-} Q } \cong \; \us {k\in K} \bigoplus \; _Q  L^2 (Q_{\alpha_k})_Q$.
So,
\[
\left( \; X \; \ol X \; \right)^{\us N \otimes k} \us {N\t {-} N }\cong \; \us {h_1,h_2,\ldots \in H}{\us {k_1,  k_2, \ldots \in K} \bigoplus} \; _N  L^2 (Q_{\alpha_{k_1} \alpha_{h_1} \alpha_{k_2} \alpha_{h_2} \cdots \alpha_{k_m}}) {} _N  \us {N\t {-} N }\cong \; \us {h_1,h_2,\ldots \in H}{\us {k_1,  k_2, \ldots \in K} \bigoplus} \; _N  L^2 (Q_{\alpha_{k_1 h_1 k_2 h_2 \cdots k_m}}) {} _N
\]
Since the subgroups $ H $ and $ K $ generate $ G $, therefore the set $\Lambda:= \left\{\left. X_g := {}_N L^2 (Q_{\alpha_g} ) {}_N \right| g \in G\right\} $ forms a full weight set in $ \mcal C_{NN} $.
It is possible to reduce the indexing set $ G $ of the weight set $ \Lambda $ further, since $ X_g \cong X_{gh} $ for all $ g \in G, h\in H $.
However, we will not do that since by reducing the weight set, one needs to work with coset representative which makes the calculations more cumbersome.

\subsection{Morphism spaces in $ \mcal C_{NN} $}$ \; $

For the affine annular algebra over $ G $ (indexing the above set), we do not need all morphism spaces of $ \mcal C_{NN} $.
We will instead concentrate on morphisms between elements of $\Lambda  $ and their tensor products.
Before that, we need more notations.
Choose a map $ u:G\times G \ra \mcal U (Q) $ such that $ \alpha_{g_1} \alpha_{g_2} = \t {Ad}_{u(g_1, g_2)} \alpha_{g_1 g_2} $ and 
\begin{equation}\label{ucond}
u \left( H\times H \; \cup \; K\times K \; \cup \; G \times \{e\} \; \cup \; \{e\} \times G \right) = \{1_Q\}.
\end{equation}
Again, associativity of multiplication in $ G $ and condition \ref{ucond} will give us a $ 3 $-cocycle $ \omega $ satisfying Equation \ref{uomega} and 
\begin{equation}\label{omegaHK1}
\left. \omega \right|_{H\times H \times H} \equiv 1 \equiv \left. \omega \right|_{K \times K \times K}.
\end{equation}
This along with Equation \ref{3coc}, implies $\omega (g, l ,l^{-1}) = \omega (gl, l^{-1}, l)$ and $\omega (l^{-1}, l , g) = \omega (l, l^{-1}, lg)$ for all $g \in G$, $l \in H \cup K$.
We will now prove a lemma on scalar cocycles which lets us choose the map $ u $ in such a way that the $ 3 $-cocycle $ \omega $ gets simplified making our calculations easy.
\begin{lem}\label{gllemma}
Any scalar $ 3 $-cocycle $\omega$ of a group $ G $ generated by subgroups $ H $ and $ K $, is coboundarily equivalent to $\omega^{\prime}$ which satisfies the relation \ref{omegaHK1} as well as
\begin{align}\label{gl}
\omega^{\prime} (g,l,l^{-1}) = 1 = \omega^{\prime} (l^{-1}, l , g) \text{ for all $g \in G$, $l \in H \cup K$.}
\end{align}
\end{lem}
\begin{proof}
Consider the subsets
$A_H = (H^\complement \times H^{\times})$,
$A_K = (K^\complement \times K^{\times})$,
$V_H = (H^{\times} \times H^\complement)$,
and
$V_K =(K^{\times} \times K^\complement)$ of
$G\times G$, and
the order $2$ bijections
$G\times G \ni (g_1,g_2) \hat{\mapsto}
(g_1,g_2)\hat{} = (g_1g_2,g^{-1}_2) \in G\times G$
and
$G\times G \ni (g_1,g_2) \check{\mapsto}
(g_1,g_2)\check{} = (g^{-1}_1,g_1g_2) \in
G\times G$
(where $H^\times = H \setminus \{e\}$ and
$K^\times = K \setminus \{e\}$). Note
that $A_H$ and $A_K$ (resp. $V_H$ and $V_K$)
are separately closed under $\hat{}$
(resp. $\check{}$ ) and have no fixed
points. Now,
$A_H \cap V_K = (K \setminus H) \times (H
\setminus K)$ (resp.
$A_K \cap V_H = (H \setminus K) \times (K
\setminus H)$) is mapped into
$A_H \setminus V_K$ (resp.
$A_K \setminus V_H$) under $\hat{}$ and into
$V_K \setminus A_H$ (resp.
$V_H \setminus A_K$) under $\check{}$ .
We choose

(i) a representative in each orbit of
$\; \hat{}$ inside ${A_H \cup A_K}$ such
that
the representative of the orbit containing
$(k,h) \in A_H \cap V_K$ is chosen as
$(kh,h^{-1})$ and the representative of the
one containing $(h,k) \in A_K \cap V_H$ is
chosen as $(hk,k^{-1})$,

(ii) a representative in each orbit of
$\; \check{}$ inside ${V_H \cup V_K}$ such
that
the representative of the orbit containing
$(k,h) \in A_H \cap V_K$ is chosen as
$(k^{-1},kh)$ and the representative of the
one containing $(h,k) \in A_K \cap V_H$ is
chosen as $(h^{-1},hk)$.

Let $A$ (resp. $V$) be the set of
representatives in $A_H \cup A_K$ (resp.
$V_H \cup V_K$). From our choice, it can be
verified that $A \cap V = \emptyset$. Define
$\vphi : G\times G \rightarrow \mathbb T$ by:

(a) $\vphi |_{G\times G \setminus (A \cup V)} = 1$,

(b)
$\vphi (g,l) = \omega (g,l,l^{-1}) = \omega
(gl, l^{-1},l)$ for $(g,l) \in A$,

(c)
$\vphi (l,g) = \omega (l^{-1},l,g) = \omega
(l, l^{-1},lg)$ for $(l,g) \in V$.

It follows that $\partial^2 ( \vphi ) $ is
normalized since $\omega$ is also so, and
$\partial^2 ( \vphi ) $ satisfies the
relation \ref{omegaHK1} since
$(H\times H \cup K \times K) \cap (A \cup V) = \emptyset$
(where $\partial^2$ denotes the $2$-cochain
map). Thus, the $3$-cocycle
$\omega^\prime = \partial^2 (\vphi) \cdot
\omega$ is normalized and satisfies relation \ref{omegaHK1}.

For relation \ref{gl}, we consider $g \in G$
and $l \in H \cup K$. Without loss of generality,
we assume $g \neq e \neq l$. So,
$\left\{ (g,l) , (gl,l^{-1}) \right\}$ (
resp. $\left\{ (l,g) , (l^{-1},lg)
\right\}$) is an orbit of $\; \hat{}$ (resp.
$\check{} \;$) in $A_H \cup A_K$ (resp.
$V_H \cup V_K$), and $\phi$ takes the value
$\omega (g , l , l^{-1}) = \omega (gl ,
l^{-1} , l)$ (resp.
$\omega (l^{-1} , l , g) = \omega (l ,
l^{-1} , lg)$) on the representative of the
orbit and $1$ on the other. This implies
\[
\omega^\prime (g, l, l^{-1}) = \ol{\vphi} (
gl,l^{-1}) \; \ol{\vphi} (g,l) \; \omega (g,
l, l^{-1}) = 1
\]
since $\vphi(g,e) = 1 = \vphi(l,l^{-1})$, and
similarly
$\omega^\prime (l^{-1} , l , g) = 1$.
\end{proof}
By the above lemma, without loss of generality, we may assume $ \omega $ satisfies:
\begin{align}\label{GL}
	\begin{tabular}{rrcl}
\text{(i)} & $\omega (g_1 ,  l , l^{-1}) $ & = \; 1 \; = &
$ \ol{\omega} (g_1 , l , l^{-1})$\\
\text{(ii)} & $\omega (g_1 , g_2 , l)$ & = &
$\ol{\omega} (g_1 , g_2 l , l^{-1})$\\
\text{(iii)} & $\omega (g_1, l, g_2)$ & = &
$\ol{\omega} (g_1 l,l^{-1}, l g_2)$\\
\text{(iv)} & $\omega (l, g_1,g_2)$ & = &
$\ol{\omega} (l^{-1}, l g_1,g_2)$\\
\text{(v)} & $u(g_1, l)$ & = &
$u^*(g_1 l , l^{-1})$\\
\text{(vi)} & $u(l ,g_2)$ & = &
$\alpha_l \left( u^* (l^{-1} , l g_2) \right)$
	\end{tabular}
\end{align}
for all $g_1, g_2 \in G$, $l \in H \cup K$ (where (ii), (iii) and (iv) are immediate implication of (i) and \ref{2coc}).
We will need the relation \ref{GL} only when $ l \in H $; however, we gave the general version, in case any reader is interested to see the actual $ 2 $-category of $ N\subset M $ instead of just $ \mcal C_{NN} $.
\begin{prop}\label{smbox}
The morphism space $ \mcal C_{NN} (X_{g_1} , X_{g_2}) $ is zero unless $ g_1 $ and $ g_2 $ give the same $ H $-$ H $ double coset, and if they do, the space has a basis given by
\[
B_{g_1,g_2} := \left\{ \left.
\begin{tabular}{l}
\\
$ X_{g_1} \ni [x]_{g_1} \os {  } \longmapsto [ \alpha_{h_1} (x) u(h_1 , g_1) u^* (g_2 , h_2)]_{g_2} \in X_{g_2} $\\
denoted by the symbol \raisebox{-1.4 em}{\1disc{$h_1 $}{$ g_1 $}{$ g_2 $}{$ h_2 $}}
\end{tabular}
\; \right| \; 
\begin{tabular}{l}
$ h_1, h_2 \in H $\\
such that \\
$ h_1 g_1 = g_2 h_2 $
\end{tabular} \right\}.
\]
\end{prop}
\begin{proof}
By Frobenius reciprocity, $ \t {dim}_\C \left( \mcal C_{NN} (X_{g_1} , X_{g_2}) \right) = \t {dim}_\C \left( \mcal C_{NN} \left( _N L^2 (N) _N \; , \; \ol X_{g_1} \us N \otimes X_{g_2}\right) \right) $.
Again $  \ol X_{g_1} \us N \otimes X_{g_2} \us {Q\t {-} Q} \cong \; {}_{\alpha_{g_1}} \left[ Q \rtimes H \right] {}_{\alpha_{g_2}}$ where the left and right actions of $ Q $ on $ Q\rtimes H $ is twisted by $ \alpha_{g_1} $ and $ \alpha_{g_2} $ respectively.
Any element of $ \mcal C_{NN} \left( _N L^2 (N) _N \; , \; {}_{\alpha_{g_1}} \left[ Q \rtimes H \right] {}_{\alpha_{g_2}}\right) $ corresponds to an element of $ Q\rtimes H $ (the image of $ \hat 1 $), say $y = \us {h\in H} \sum y_h \; h$.
By $ N $-$ N $ linearity, we will have $ \alpha_{g_1} (n) \; y_h = y_h \; \alpha_h (\alpha_{g_2} (n)) $ for all $  n\in N, h\in H $, equivalently
\[
n \; \alpha^{-1}_{g_1} (y_h) = \alpha^{-1}_{g_1} (y_h) \; \alpha^{-1}_{g_1}(\alpha_h (\alpha_{g_2} (n))) \t{ for all }  n\in N, h\in H.
\]

The following is a well-known fact for the fixed-point subfactor $ N\subset Q $ of an outer action of $ H $. For $ y \in Q $ and $ \theta \in \t {Aut} (Q) $, the following are equivalent:

(i) $ y \neq 0 $ and $ n y = y \theta (n) $ for all $ n \in N = Q^H $,

(ii) $ y_0 := \displaystyle \frac {y}{\norm{y}} \in \mcal U (Q)$ and $ \t{Ad}_{y_0} \circ \theta \in \{\alpha_h : h \in H\} $.

By the above fact, $ y \neq 0 $ only when there exists $ h_1 , h_2  \in H$ such that $ \alpha^{-1}_{g_1} \alpha_{h_1} \alpha_{g_2} \alpha_{h_2} \in \t {Inn} (Q)$, equivalently $ g_1 $ and $ g_2 $ generate the same $ H $-$ H $ double coset.
In particular, $ y_h = 0$ unless $ h $ belongs to $ H \cap g_1 H g^{-1}_2 $.
And for $ h \in  H \cap g_1 H g^{-1}_2$, for $ y_h \neq 0 $, we have $ \t {Ad}_{\alpha^{-1}_{g_1 } (y_0) } \; \alpha^{-1}_{g_1} \; \alpha_h \; \alpha_{g_2} = \alpha_{g^{-1}_1 h g_2} $ where $ y_0 = \displaystyle \frac{y_h}{\norm {y_h}} $.
This implies $  \t {Ad}_{y_0} \; \alpha_h \; \alpha_{g_2} = \alpha_{g_1} \; \alpha_{g^{-1}_1 h g_2} $, equivalently, $  \t {Ad}_{y_0 u(h ,g_2) } = {Ad}_{u(g_1 , g^{-1}_1 h g_2) } $.
Hence, $ y_h \in \C \{u(g_1 , g^{-1}_1 h g_2) u^*(h ,g_2)\}$.
Thus, the set
\[
\left\{ \left(u (g_1 , h^{-1}_2 ) u^* (h^{-1}_1 ,g_2)\right) \; h^{-1}_1 : h_1, h_2 \in H \t { such that } h_1 g_1 = g_2 h_2 \right\} 
\]
forms a basis of the vector space $ V:= \left\{ y \in Q \rtimes H : \alpha_{g_1} (n) y = y \alpha_{g_2} (n) \t { for all } n \in N \right\} $.
To show that the set $ B_{g_1,g_2} $ forms a basis for $ Hg_1 H = H g_2 H $, we need the following explicit isomorphism:
\[
V \ni y \os {\pi} \longmapsto \pi(y) := J y^* J \in  \mcal C_{NN} (X_{g_1} , X_{g_2})
\]
where $ J $ is the canonical anti-unitary of $ L^2 (Q) $.
Set $ y_{h_1,h_2} : = \left(u (g_1 , h^{-1}_2 ) u^* (h^{-1}_1 ,g_2)\right) \; h^{-1}_1 $ for $ h_1 g_1 = g_2 h_2 $.
Then,
\[
\pi (y_{h_1,h_2}) [x]_{g_1} = \left[ \alpha_{h_1} \left( x u (g_1 , h^{-1}_2 ) u^* (h^{-1}_1 ,g_2 ) \right) \right]_{g_2} = \left[ \alpha_{h_1} \left( x \right) \alpha_{h_1} \left(u (g_1 , h^{-1}_2 ) u^* (h^{-1}_1 ,g_2 ) \right) \right]_{g_2}.
\]
We simplify $ \alpha_{h_1} \left(u (g_1 , h^{-1}_2 ) u^* (h^{-1}_1 ,g_2 ) \right) $ using Equations \ref{uomega}, \ref{ucond} and \ref{GL} to get
\begin{align*}
& \; \left\{ \ol \omega (h_1,g_1 , h^{-1}_2) u(h_1,g_1) u (h_1g_1 ,h^{-1}_2) u^*(h_1,g_1 h^{-1}_2)\right\} \; u(h_1 , h^{-1}_1 g_2) \\
= & \; \ol \omega (h_1,g_1 , h^{-1}_2) \; u(h_1,g_1) u (g_2 h_2 ,h^{-1}_2) = \; \ol \omega (h_1,g_1 , h^{-1}_2) \; \left\{u(h_1,g_1) u^* (g_2, h_2)\right\}.
\end{align*}
Hence, $ \pi (y_{h_1,h_2}) $ is a unit scalar multiple of \raisebox{-1.5 em}{\1disc{$h_1 $}{$ g_1 $}{$ g_2 $}{$ h_2 $}}
corresponding to $ (h_1,h_2) $.
\end{proof}
\begin{rem}\label{duality}
The maps
\[
{}_N L^2(N)_N \ni \hat 1 \os {\displaystyle R_g} \longmapsto \us i \sum [u^*(g^{-1} , g) \alpha_{g^{-1}} (b_i) ]_{g^{-1}} \us N \otimes [b^*_i]_g \in X_{g^{-1}} \us N \otimes X_g
\]
\[
{}_N L^2(N)_N \ni \hat 1 \os {\displaystyle \ol R_g} \longmapsto \omega (g,g^{-1},g) \us i \sum [u^*(g, g^{-1}) \alpha_{g} (b_i) ]_{g} \us N \otimes [b^*_i]_{g^{-1}} \in X_{g} \us N \otimes X_{g^{-1}}
\]
are standard solutions to conjugate equations for duality of $ X_g $ where $ \{b_i\}_i $ is a basis for the subfactor $ N \subset Q $.
We will also need the $ * $ of these maps, namely
\[
X_{g^{-1}} \us N \otimes X_g \ni [x]_{g^{-1}} \us N \otimes [y]_g \os {\displaystyle R^*_g} \longmapsto E_N \left(x \alpha_{g^{-1}} (y) u(g^{-1} , g )\right) \in {}_N L^2(N)_N
\]
\[
 X_{g} \us N \otimes X_{g^{-1}} \ni [x]_g \us N \otimes [y]_{g^{-1}} \os {\displaystyle {\ol R}^*_g} \longmapsto \ol \omega (g, g^{-1}, g) \; E_N \left(x \alpha_{g} (y) u(g , g^{-1} )\right) \in {}_N L^2(N)_N
\]
\end{rem}
\begin{prop}\label{smbox*alg}$ { } $

(i) \hspace{1em} \raisebox{-2.7 em}
{\psfrag{1}{\reflectbox{$h_3 $}}
\psfrag{2}{$ g_2 $}
\psfrag{3}{$ g_3 $}
\psfrag{4}{$ h_4 $}
\psfrag{5}{\reflectbox{$h_1 $}}
\psfrag{6}{$ g_1 $}
\psfrag{7}{$ h_2 $}
\includegraphics[scale=0.2]{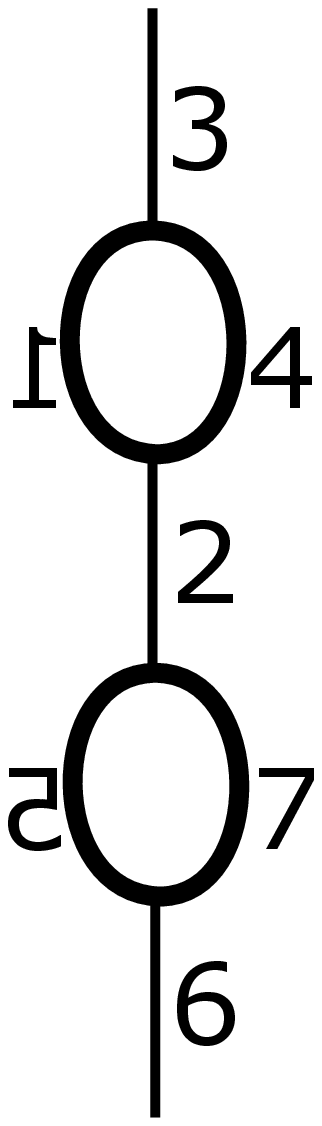}} \; $ := \; \raisebox{-1.5 em}{\1disc{$h_3 $}{$ g_2 $}{$ g_3 $}{$ h_4 $}} \; \circ \; \raisebox{-1.5 em}{\1disc{$h_1 $}{$ g_1 $}{$ g_2 $}{$ h_2 $}} \; = \; \ol \omega (h_3, h_1, g_1) \; \omega (h_3, g_2, h_2) \; \ol \omega (g_3 , h_4 , h_2)$ \;\;\;  \raisebox{-1.5 em}{\1disc{$ h_3 h_1 $}{$ g_1 $}{$ g_3 $}{$ h_4 h_2 $}}

(ii) $  \left[ \raisebox{-1.5 em}{\1disc{$\!\!\! h_1 $}{$ g_1 $}{$ g_2 $}{$ h_2 $}} \right]^* \; = \; \ol \omega (h_1, g_1, h^{-1}_2)  $
\; \raisebox{-1.5 em}{\1disc{$\!\!\!\!\!\! h^{-1}_1 $}{$ g_2 $}{$ g_1 $}{$ h^{-1}_2 $}}
\end{prop}
\begin{proof}
(i) The left side is given by $ [x]_{g_1} \mapsto \left[\alpha_{h_3 h_1} (x) \; \alpha_{h_3} (u(h_1,g_1) u^*(g_2,h_2)) \;  u(h_3,g_2)u^*(g_3,h_4)\right]_{g_3}$.
Observe that
\begin{align*}
& \; \alpha_{h_3} (u(h_1,g_1) u^*(g_2,h_2)) \;  u(h_3,g_2)u^*(g_3,h_4)\\
= & \; \ol \omega (h_3, h_1,g_1) \; u(h_3 h_1,g_1) \; u^*(h_3 , \us {=g_2 h_2} {h_1 g_1}) \;  \alpha_{h_3} (u^*(g_2,h_2)) \;  u(h_3,g_2) \; u^*(g_3,h_4)\\
& \; \t {(applying \ref{uomega} and \ref{ucond} on the first term)}\\
= & \; \ol \omega (h_3, h_1,g_1) \; \omega (h_3,g_2,h_2) \; u(h_3 h_1,g_1) \; u^*(\us {=g_3h_4}{h_3 g_2} , h_2) \; u^*(g_3,h_4)\\
& \; \t {(applying \ref{uomega} on the second, third and fourth terms)}\\
= & \; \ol \omega (h_3, h_1,g_1) \; \omega (h_3,g_2,h_2) \; \ol \omega (g_3, h_4,h_2) \; u(h_3 h_1 , g_1) \; u^*(g_3,h_4h_2)\\
& \; \t {(applying \ref{uomega} and \ref{ucond} on the last two terms)}
\end{align*}
which gives the required result.

(ii) Note that \raisebox{-1.4 em}{\1disc{$\!\!\! h_1 $}{$ g_1 $}{$ g_2 $}{$ h_2 $}} is a unitary which follows right from its definition.
Using part (i), one can easily show that \; \raisebox{-1.4 em}{\1disc{$\!\!\!\!\!\! h^{-1}_1 $}{$ g_2 $}{$ g_1 $}{$ h^{-1}_2 $}} \; is indeed the inverse of \raisebox{-1.4 em}{\1disc{$\!\!\! h_1 $}{$ g_1 $}{$ g_2 $}{$ h_2 $}} where one uses the relations in \ref{GL}.
\end{proof}
Next, we will prove some facts about tensor product of two elements from $ \Lambda $.
For $ g_1,g_2 \in G$ and $ h\in H $, we define \raisebox{-1.5em}{\tower{$ g_1 $}{$ g_2 $}{$ g_1hg_2 $}{$ h $}}  $ : X_{g_1} \us N \otimes X_{g_2} \ra X_{g_1hg_2}$ in the following way
\[
X_{g_1} \us N \otimes X_{g_2} \ni [x]_{g_1} \us N \otimes [y]_{g_2} \longmapsto \abs {H}^{-\frac 1 2} \left[ x \; \alpha_{g_1} (\alpha_h (y)) \; u(g_1 ,h) \; u(g_1 h, g_2) \right]_{g_1h g_2} \in X_{g_1hg_2}.
\]
\begin{rem}
With standard inner product computation, one can show that
\begin{align*}
\left( \raisebox{-1.5em}{\tower{$ g_1 $}{$ g_2 $}{$ g_1hg_2 $}{$ h $}}  \right)^* : [z]_{g_1hg_2} \longmapsto & \; \abs{H}^{- \frac 1 2} \us {i} \sum [z u^*(g_1h ,g_2) u^* (g_1 , h) \alpha_{g_1} (b_i)]_{g_1} \us N \otimes [\alpha_{h^{-1}} (b^*_i)]_{g_2}\\
= & \; \abs{H}^{- \frac 1 2} \us {i} \sum [\alpha_{g_1} (b_i)]_{g_1} \us N \otimes [\alpha_{h^{-1}} \left(b^*_i \alpha^{-1}_{g_1} (z u^*(g_1h ,g_2) u^* (g_1 , h) ) \right) ]_{g_2}
\end{align*}
where $ \{b_i\}_i $ is \textbf{any} basis of $ Q $ over $ N $.
\end{rem}
\noindent To see this, consider $ \left\lab \raisebox{-1.5em}{\tower{$ g_1 $}{$ g_2 $}{$ g_1hg_2 $}{$ h $}}  \left([x]_{g_1} \us N \otimes [y]_{g_2} \right) \; , \; [z]_{g_1hg_2} \right \rab $
\begin{align*}
&= \abs {H}^{- \frac 1 2} \t {tr} \left( x \; \alpha_{g_1} (\alpha_h (y)) \; u(g_1 ,h) \; u(g_1 h, g_2) \; z^* \right)\\
&= \abs {H}^{- \frac 1 2} \us i \sum \t {tr} \left( x \; \alpha_{g_1} \left( E_N (y \alpha_{h^{-1}} ( b_i) ) b^*_i \right) \; u(g_1 ,h) \; u(g_1 h, g_2) \; z^* \right)\\
& = \abs {H}^{- \frac 1 2} \us i \sum \t {tr} \left( x \; \alpha_{g_1} \left( E_N (y \alpha_{h^{-1}} ( b_i) ) \right) \; \left(z \; u^*(g_1 h, g_2) \; u^*(g_1 ,h) \; \alpha_{g_1} (b_i) \right)^* \right)\\
& = \abs {H}^{- \frac 1 2} \us i \sum  \left\lab [x]_{g_1} \;  _N \left\lab [y]_{g_2} , [\alpha_{h^{-1}} ( b^*_i)]_{g_2} \right\rab  \; , \; \left[z \; u^*(g_1 h, g_2) \; u^*(g_1 ,h) \; \alpha_{g_1} (b_i) \right]_{g_1} \right\rab
\end{align*}
$= \left\lab [x]_{g_1} \us N \otimes [y]_{g_2} \; , \; \left(\raisebox{-1.5em}{\tower{$ g_1 $}{$ g_2 $}{$ g_1hg_2 $}{$ h $}}  \right)^* [z]_{g_1hg_2} \right \rab $.
We will denote $\left( \raisebox{-1.5em}{\tower{$ g_1 $}{$ g_2 $}{$ g_1hg_2 $}{$ h $}} \right)^*$ by \raisebox{-1.5em}{\wine{$ g_1 $}{$ g_2 $}{$ g_1hg_2 $}{$ h $}}.
It is straightforward to check \raisebox{-1.5em}{\wine{$ g_1 $}{$ g_2 $}{$ g_1hg_2 $}{$ h $}} preserves inner product and thereby is an isometry.
So, the element
\;\;\raisebox{-3em}{
	\psfrag{1}{\reflectbox{$ g_1 $}}
	\psfrag{2}{$ g_2 $}
	\psfrag{3}{$g_1 h g_2 $}
	\psfrag{4}{$ h $}
	\psfrag{5}{\reflectbox{$ g_1 $}}
	\psfrag{6}{$ g_2 $}
	\psfrag{7}{$ h $}
	\includegraphics[scale=0.2]{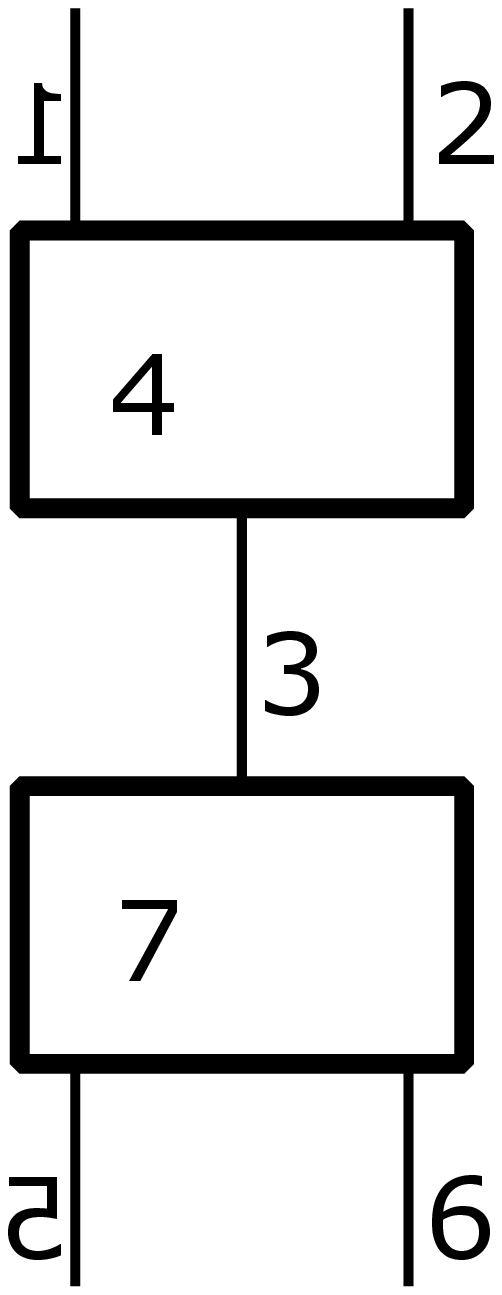}
	}\hspace{1em}
$ := \left( \raisebox{-1.5em}{\wine{$ g_1 $}{$ g_2 $}{$ g_1hg_2 $}{$ h $}} \; \circ \; \raisebox{-1.5em}{\tower{$ g_1 $}{$ g_2 $}{$ g_1hg_2 $}{$ h $}} \; \right) $ is a projection in $ \t {End } (X_{g_1} \us N \otimes X_{g_2}) $ for every $ h\in H $.
\begin{prop}\label{resid}
The set $ \left\{ \raisebox{-1.5em}{\tower{$ g_1 $}{$ g_2 $}{$ g_1hg_2 $}{$ h $}} : h \in H \right\} $ gives a resolution of the identity in $ \t {End } (X_{g_1} \us N \otimes X_{g_2}) $.
\end{prop}
\begin{proof}
It is enough to check  $\displaystyle \us {h\in H} \sum  \raisebox{-3em}{
	\psfrag{1}{\reflectbox{$ g_1 $}}
	\psfrag{2}{$ g_2 $}
	\psfrag{3}{$g_1 h g_2 $}
	\psfrag{4}{$ h $}
	\psfrag{5}{\reflectbox{$ g_1 $}}
	\psfrag{6}{$ g_2 $}
	\psfrag{7}{$ h $}
	\includegraphics[scale=0.2]{figures/cnn/winetower.eps}
} \hspace{1em} = \; \t {id}_{X_{g_1} \us N \otimes X_{g_2}}$.
The left side acting on $ [x]_{g_1} \us N \otimes [y]_{g_2} $ gives
\begin{align*}
= & \; \abs{H}^{- 1} \us {i,h} \sum [\alpha_{g_1} (b_i)]_{g_1} \us N \otimes [\alpha_{h^{-1}} \left(b^*_i \alpha^{-1}_{g_1} ( x \; \alpha_{g_1} (\alpha_h (y)) ) \right) ]_{g_2}\\
= & \; \abs{H}^{- 1} \us {i,h} \sum [\alpha_{g_1} (b_i)]_{g_1} \us N \otimes [\alpha_{h^{-1}} \left(b^*_i \alpha^{-1}_{g_1} ( x  )  \right) \; y]_{g_2} =  \us {i} \sum [\alpha_{g_1} (b_i)]_{g_1} \us N \otimes [E_N \left(b^*_i \alpha^{-1}_{g_1} ( x  )  \right) \; y]_{g_2} =  [x]_{g_1} \us N \otimes [y]_{g_2}.
\end{align*}
\end{proof}
\begin{rem}\label{hexagon}
From Propositions \ref{smbox} and \ref{resid}, we may conclude that  $ \mcal C_{NN} (X_{g_1} \us N \otimes X_{g_2} , X_{g_3} \us N \otimes X_{g_4}) $ is linearly spanned by the (linearly independent) set
\[
\left\{ \left. \raisebox{-4.4em}{\hexagon{$ g_1 $}{$ g_2 $}{$h_3 $}{$ g_3 $}{$ g_4 $}{$ h_4 $}{$ h_1 $}{$h_2  $}{\psfrag{9}{}\psfrag{0}{}}}
\; := \; \raisebox{-1.5em}{\wine{$ g_3 $}{$ g_4 $}{$ g_3 h_4 g_4 $}{$ h_4 $}} \; \circ \; \raisebox{-1.5 em}{\1disc{$h_1 $}{$ g_1 h_3 g_2 $}{$ g_3 h_4 g_4 $}{$ h_2 $}} \hspace{1em} \circ \raisebox{-1.5em}{\tower{$ g_1 $}{$ g_2 $}{$ g_1h_3g_2 $}{$ h_3 $}} \hspace{1em} \right| h_1,h_2,h_3 ,h_4 \in H \right\}.
\]
\end{rem}
\noindent We will now prove two lemmas which will be very useful in finding the structure the annular algebra.
As for notations, we will use the standard graphical representations of morphism where composition will be represented by stacking the morphisms vertically with the left most being in the top.
\begin{lem}\label{tribox}
\begin{align*}
	\t{(i)} \hspace{2em}& \raisebox{-3em}{\psfrag{1}{\reflectbox{$ g_1 $}}
		\psfrag{2}{$ s $}
		\psfrag{3}{$ g_1 h  s $}
		\psfrag{4}{$ h $}
		\psfrag{5}{\reflectbox{$ h_1 $}}
		\psfrag{6}{$ h_2 $}
		\psfrag{7}{$ t $}
		\includegraphics[scale=0.2]{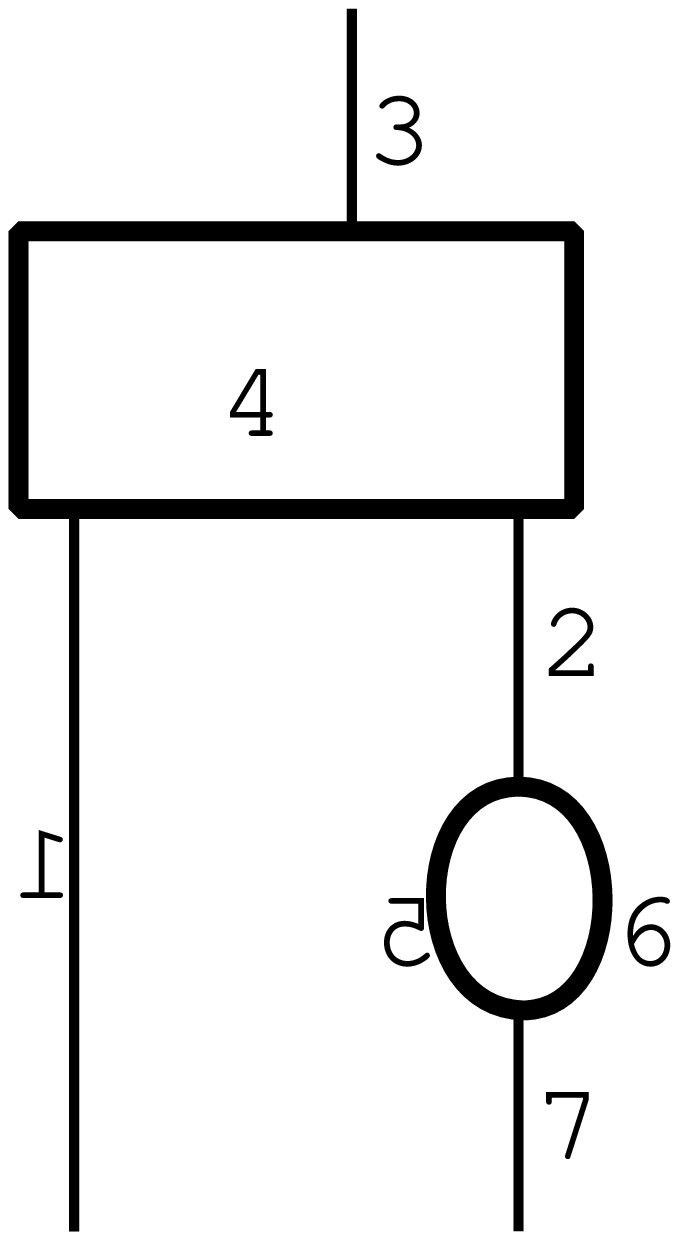}}
\hspace{1em} =  \left[ \omega (g_1 h , s , h_2) \; \ol \omega (g_1 h , h_1 ,t) \; \omega (g_1 ,  h,h_1) \right] \; \raisebox{-3em}{
	\psfrag{1}{\reflectbox{$ g_1 $}}
	\psfrag{2}{$ t $}
	\psfrag{3}{$ g_1 hh_1 t $}
	\psfrag{4}{$ h h_1 $}
	\psfrag{5}{\reflectbox{$ e \;$}}
	\psfrag{6}{$ h_2 $}
	\psfrag{7}{$ g_1 h s $}
	\includegraphics[scale=0.2]{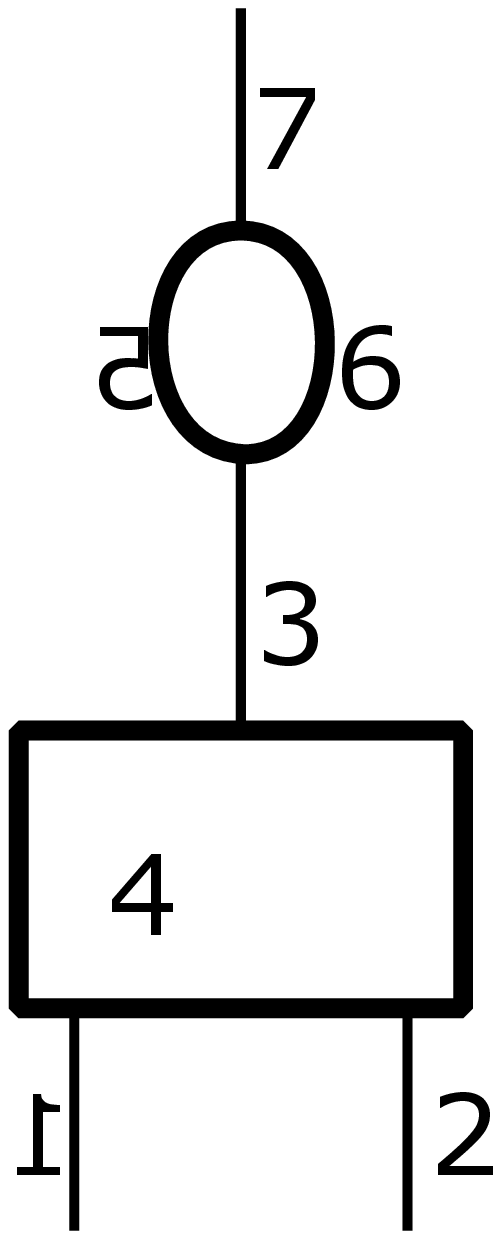}}
\\
\t{(ii)} \hspace{2em} &\raisebox{-3em}{
	\psfrag{1}{\reflectbox{$ s $}}
	\psfrag{2}{$ g_2 $}
	\psfrag{3}{$ s h g_2 $}
	\psfrag{4}{$ h $}
	\psfrag{5}{\reflectbox{$ h_1 $}}
	\psfrag{6}{$ h_2 $}
	\psfrag{7}{$ t $}
	\includegraphics[scale=0.2]{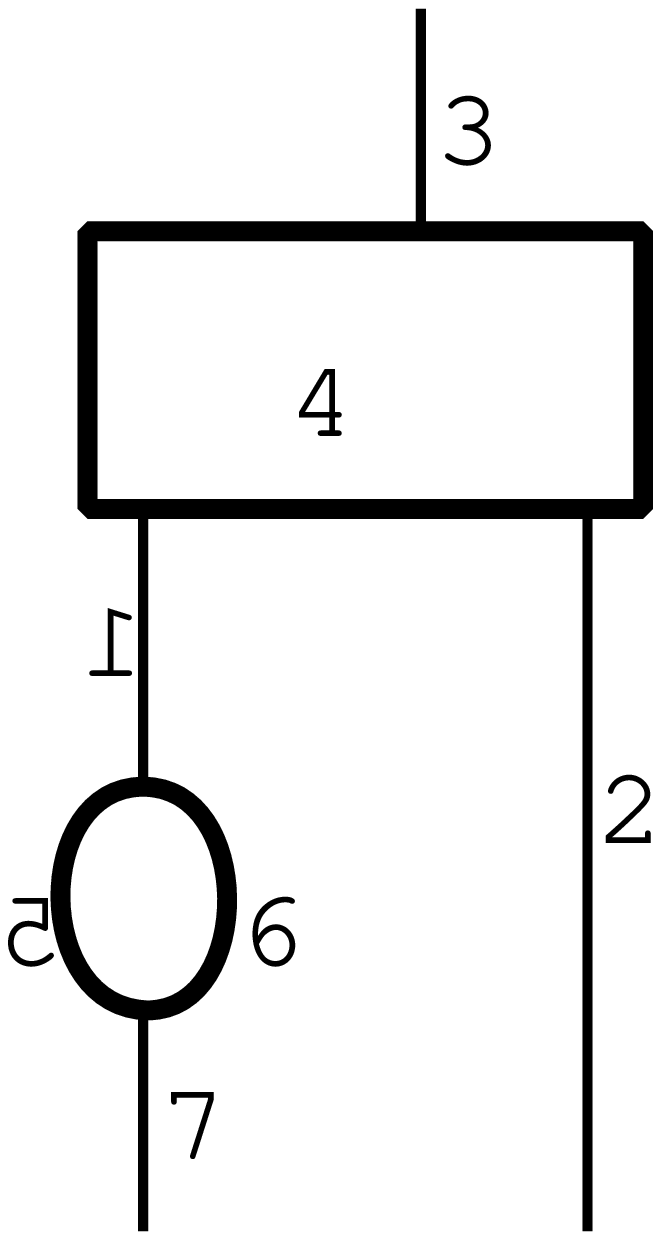}}
\hspace{1em} = \left[ \omega (h_1 , t h^{-1}_2 h , g_2)\; \ol \omega (s, h_2 , h^{-1}_2 h) \; \omega (h_1 , t , h^{-1}_2 h) \right] \; \raisebox{-3em}{
	\psfrag{1}{\reflectbox{$ t $}}
	\psfrag{2}{$ g_2 $}
	\psfrag{3}{$ t h^{-1}_2 h  g_2 $}
	\psfrag{4}{\!\!\!$ h^{-1}_2 h $}
	\psfrag{5}{\reflectbox{$ h_1$}}
	\psfrag{6}{$ e $}
	\psfrag{7}{$s h g_2 $}
	\includegraphics[scale=0.2]{figures/cnn/triboxirhs.eps}}
\end{align*}
\end{lem}
\begin{proof}
	(i) The left side acts on $ [x]_{g_1} \us N \otimes [y]_t $, gives
	\begin{align*}
		\abs{H}^{- \frac 1 2} \left[x \alpha_{g_1} \left(\alpha_h \left(\alpha_{h_1} (y) u(h_1 ,t) u^*(s,h_2)\right)\right) u(g_1,h) u(g_1 h , s) \right]_{g_1h s}
	\end{align*}
	whereas the right side yields
	\[
	\abs{H}^{- \frac 1 2} \left[x \alpha_{g_1} \left(\alpha_{h h_1} (y) \right) u(g_1 , hh_1) u(g_1 h h_1, t)  u^*(g_1 h s , h_2) \right]_{g_1h s}.
	\]
	After striking out the similar terms, we will be left with
	\begin{align*}
		& \; \alpha_{g_1} \left(\alpha_h \left( u(h_1 ,t) u^*(s,h_2)\right)\right) u(g_1,h) u(g_1 h , s)\\
		= & \; u(g_1,h) \alpha_{g_1 h} \left( u(h_1 ,t) u^*(s,h_2)\right)  u(g_1 h , s)\\
		= & \;  u(g_1,h) \alpha_{g_1 h} \left( u(h_1 ,t) \right) \omega (g_1 h , s , h_2) u(g_1h, \us {=h_1 t} {s h_2}) u^* (g_1 h s ,h_2) \\
		= & \; \omega (g_1 h , s , h_2) u(g_1,h) \ol \omega (g_1 h , h_1 ,t) u(g_1 h, h_1) u(g_1h h_1 , t) u^* (g_1 h s ,h_2) \\
		= & \; \left[\omega (g_1 h , s , h_2) \ol \omega (g_1 h , h_1 ,t) \omega (g_1 ,  h,h_1)\right] \; \left( u(g_1 , hh_1) u(g_1 h h_1, t)  u^*(g_1 h s , h_2) \right)
	\end{align*}
	
	(ii) The action of left side on $ [x]_t \us N \otimes [y]_{g_2} $ is
	\begin{align*}
		& \; \abs H^{- \frac 1 2} \left[\alpha_{h_1} (x) u(h_1 , t) u^* (s ,h_2) \; \alpha_s \left( \alpha_h (y)  \right) \; u(s,h) u(sh,g_2)\right]_{shg_2}\\
		= & \; \abs H^{- \frac 1 2} \left[\alpha_{h_1} \left(\; x \; \alpha_t (\alpha_{h^{-1}_2 h} (y) \;  \right) u(h_1 , t) u^* (s ,h_2) \;  u(s,h) u(sh,g_2) \right]_{shg_2}
	\end{align*}
	and the right side on the same is
	\begin{align*}
		& \; \abs H^{- \frac 1 2}  \left[\alpha_{h_1} \left(\; x \; \alpha_t (\alpha_{h^{-1}_2 h} (y) \; u(t, h^{-1}_2 h) u(t h^{-1}_2 h, g_2)  \right) u(h_1 , t h^{-1}_2 h g_2) \right]_{shg_2}\\
		= & \; \abs H^{- \frac 1 2}  \left[\alpha_{h_1} \left(\; x \; \alpha_t (\alpha_{h^{-1}_2 h} (y) \right) \; \alpha_{h_1} (u(t, h^{-1}_2 h) ) \; \ol \omega (h_1 , t h^{-1}_2 h , g_2) u (h_1 , t h^{-1}_2 h ) u (h_1  t h^{-1}_2 h , g_2) \right]_{shg_2}\\
		= & \; \ol \omega (h_1 , t h^{-1}_2 h , g_2) \abs H^{- \frac 1 2}  \left[\alpha_{h_1} \left(\; x \; \alpha_t (\alpha_{h^{-1}_2 h} (y) \right) \ol \omega (h_1 , t , h^{-1}_2 h) u (h_1 , t ) u (\us {=s h_2}{h_1  t} , h^{-1}_2 h) \;  u (s h , g_2) \right]_{shg_2}\\
		= & \; \ol \omega (h_1 , t h^{-1}_2 h , g_2)  \ol \omega (h_1 , t , h^{-1}_2 h) \abs H^{- \frac 1 2}\\
		& \;\left[\alpha_{h_1} \left(\; x \; \alpha_t (\alpha_{h^{-1}_2 h} (y) \right) u (h_1 , t ) \; \omega (s, h_2 , h^{-1}_2 h) u^* (s ,h_2) \;  u(s,h) \;  u (s h , g_2) \right]_{shg_2}.
	\end{align*}
\end{proof}
\begin{lem}\label{2tri}$ $\\
(i) \; \raisebox{-3em}{
	\psfrag{1}{\reflectbox{$ g_1 $}}
	\psfrag{2}{$ g_2 $}
	\psfrag{3}{$ g_3 $}
	\psfrag{4}{$ h_1 $}
	\psfrag{5}{$ h_2$}
	\psfrag{6}{\reflectbox{$ g_1 h_1 g_2 $}}
	\psfrag{7}{$g_2 h_2 g_3 $}
	\includegraphics[scale=0.2]{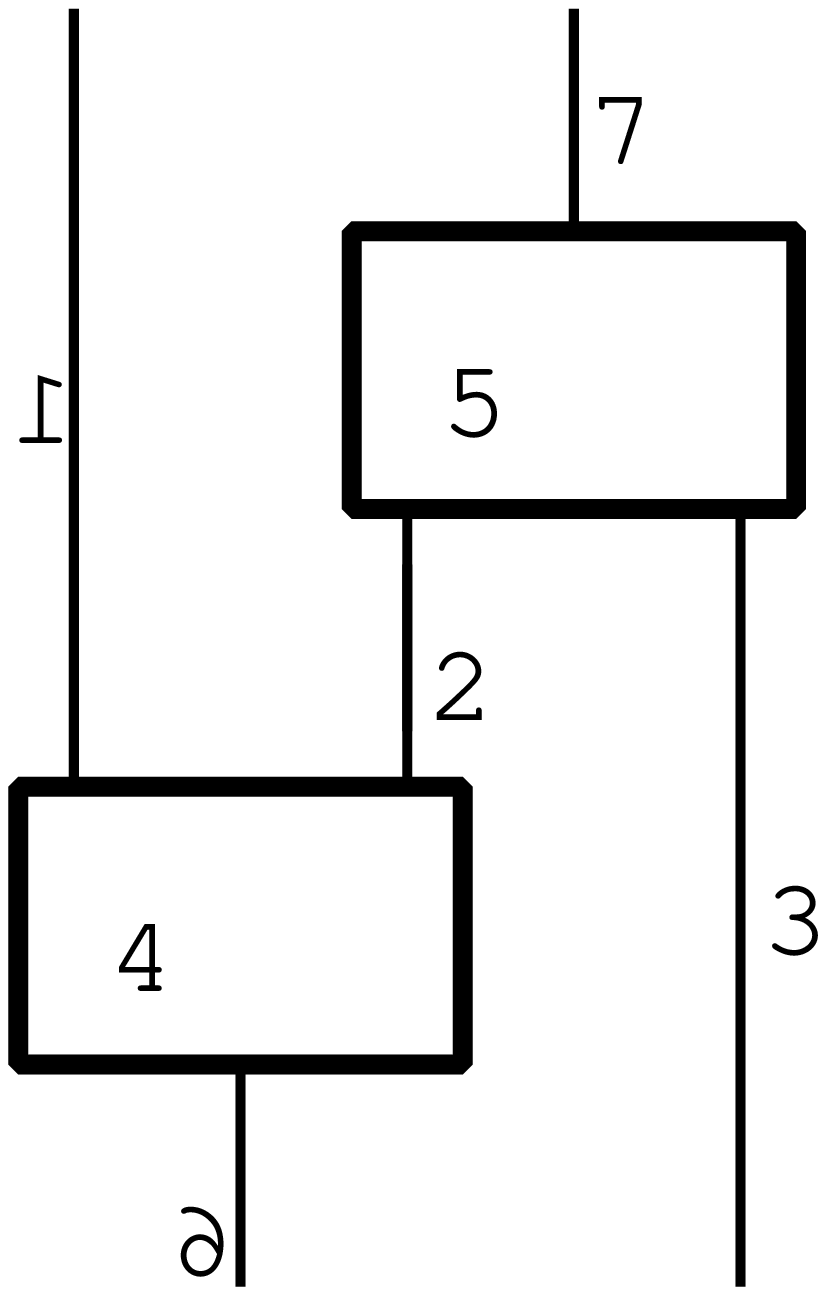}}
\; $ = \left[\ol \omega(g_1h_1, g_2 , h_2) \; \ol \omega (g_1h_1 ,g_2 h_2, g_3)\right] \hspace{1em} \raisebox{-3em}{
	\psfrag{1}{\reflectbox{$ g_1 $}}
	\psfrag{2}{$ g_2 h_2 g_3 $}
	\psfrag{3}{$g_1 h_1 g_2 h_2 g_3 $}
	\psfrag{4}{$ h_1 $}
	\psfrag{5}{\reflectbox{$ g_1 h_1g_2 $}}
	\psfrag{6}{$ g_3 $}
	\psfrag{7}{$ h_2 $}
	\includegraphics[scale=0.2]{figures/cnn/winetower.eps}
} $\\
	
\noindent(ii) \; \raisebox{-3em}{
	\psfrag{1}{\reflectbox{$ g_1 $}}
	\psfrag{2}{$ g_2 $}
	\psfrag{3}{$ g_3 $}
	\psfrag{4}{$ h_1 $}
	\psfrag{5}{$ h_2$}
	\psfrag{6}{\reflectbox{$ g_1 h_1 g_2 $}}
	\psfrag{7}{$g_2 h_2 g_3 $}
	\includegraphics[scale=0.2]{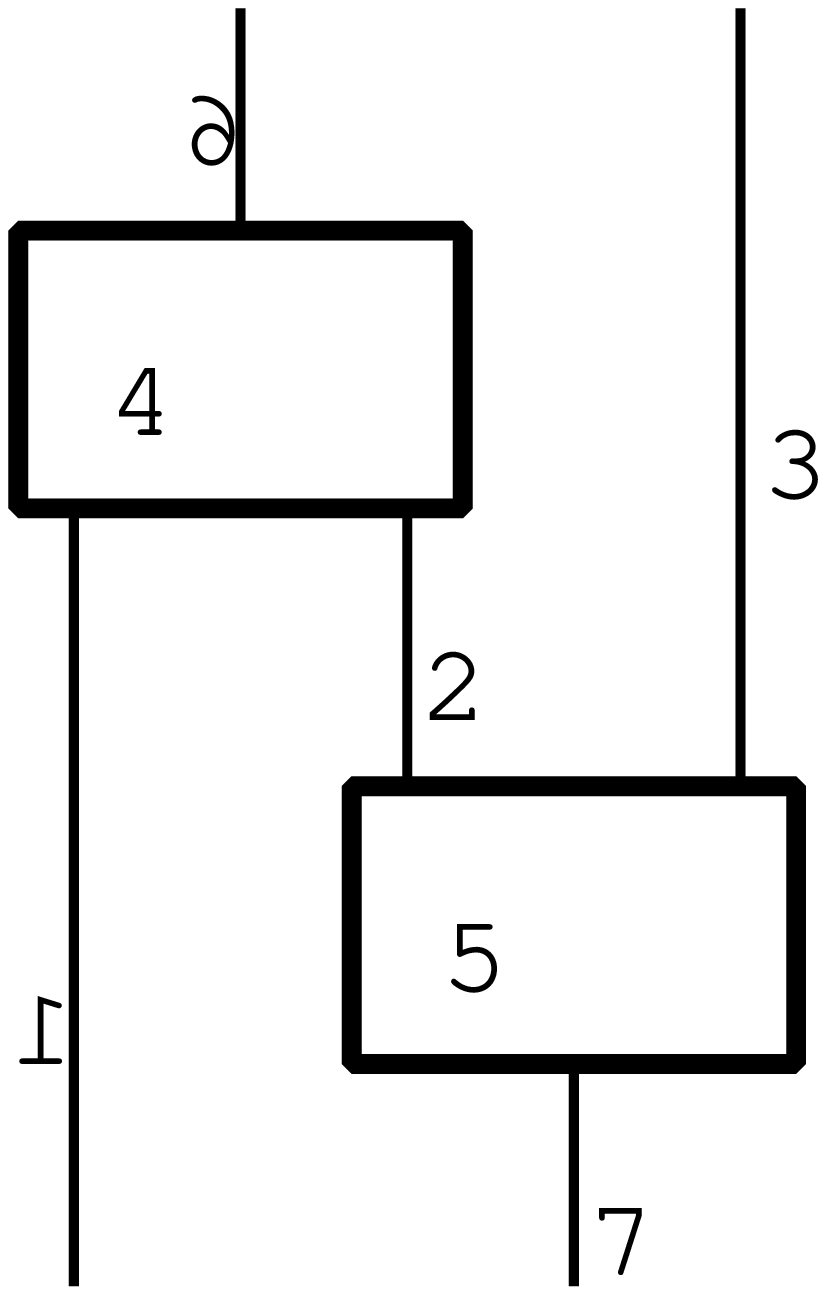}}
$ = \left[\omega(g_1h_1, g_2 , h_2) \;  \omega (g_1h_1 ,g_2 h_2, g_3)\right] \hspace{1em} \raisebox{-3em}{
	\psfrag{1}{\reflectbox{$ g_1 h_1g_2 $}}
	\psfrag{2}{$ g_3 $}
	\psfrag{3}{$g_1 h_1 g_2 h_2 g_3 $}
	\psfrag{4}{$ h_2 $}
	\psfrag{5}{\reflectbox{$ g_1 $}}
	\psfrag{6}{$ g_2 h_2 g_3 $}
	\psfrag{7}{$ h_1 $}
	\includegraphics[scale=0.2]{figures/cnn/winetower.eps}
}$
\end{lem}
\begin{proof}
The left side acting on $ [x]_{g_1 h_1 g_2} \us N \otimes [y]_{g_3}$ gives
\begin{align*}
& \; \abs{H}^{-1} \us i \sum [x u^*(g_1 h_1 ,g_2) u^* (g_1, h_1) \alpha_{g_1} (b_i)]_{g_1} \us N \otimes [\alpha^{-1}_{h_1} (b^*_i) \alpha_{g_2} (\alpha_{h_2} (y))\;  u(g_2 , h_2) u(g_2 h_2 , g_3) ]_{g_2h_2 g_3}\\
=& \; \abs{H}^{-1} \us {i,j} \sum [x u^*(g_1 h_1 ,g_2) u^* (g_1, h_1) \alpha_{g_1} (b_i)]_{g_1}\\
& \; \us N \otimes [E_N \left(b^*_i \alpha_{h_1}  \left(\alpha_{g_2} (\alpha_{h_2} (y))\;  u(g_2 , h_2) u(g_2 h_2 , g_3)\right) b_j \right) \alpha^{-1}_{h_1} (b^*_j)]_{g_2h_2 g_3}\\
=& \; \abs{H}^{-1} \us {j} \sum [x u^*(g_1 h_1 ,g_2) u^* (g_1, h_1) \alpha_{g_1} \left(\alpha_{h_1}  \left(\alpha_{g_2} (\alpha_{h_2} (y))\;  u(g_2 , h_2) u(g_2 h_2 , g_3)\right) b_j \right)]_{g_1} \us N \otimes [\alpha^{-1}_{h_1} (b^*_j)]_{g_2h_2 g_3}\\
=& \; \abs{H}^{-1} \us {j} \sum [x \; \alpha_{g_1h_1g_2} (\alpha_{h_2} (y)) \; \ul{u^*(g_1 h_1 ,g_2) u^* (g_1, h_1) \; \alpha_{g_1} \left( \alpha_{h_1} \left(  u(g_2 , h_2) u(g_2 h_2 , g_3) \right)  \right)} \; \alpha_{g_1} (b_j)]_{g_1}\\
& \; \us N \otimes [\alpha^{-1}_{h_1} (b^*_j)]_{g_2h_2 g_3}.
\end{align*}
Simplifying the underlined expression, we get
\begin{align*}
& \; u^*(g_1 h_1 ,g_2)  \; \alpha_{g_1 h_1}  \left(  u(g_2 , h_2) u(g_2 h_2 , g_3) \right) \; u^* (g_1, h_1)\\
= & \; \ol \omega(g_1h_1, g_2 , h_2)  \; u(g_1 h_1 g_2 , h_2) u^* (g_1 h_1,g_2h_2) \; \alpha_{g_1 h_1}  \left( u(g_2 h_2 , g_3) \right) \; u^* (g_1, h_1)\\
= & \; \ol \omega(g_1h_1, g_2 , h_2) \; u(g_1 h_1 g_2 , h_2) \; \ol \omega (g_1h_1 ,g_2 h_2, g_3) u(g_1h_1g_2h_2 , g_3) u^*(g_1h_1, g_2h_2g_3) \; u^* (g_1, h_1).
\end{align*}
This is exactly what we wanted from the right side acting on $ [x]_{g_1 h_1 g_2} \us N \otimes [y]_{g_3}$.

(ii) This follows from taking $ * $ on both sides.
\end{proof}
\subsection{The  affine annular algebra over the weight set $ \Lambda $ indexed by $ G $}{$  $}

Let $ \mcal A $ denote the affine annular algebra of $ \mcal C_{NN} $ with respect to $ G $ which indexes the weight set $ \Lambda $.
In our set up, the indexing set $ G $ is more important rather than the set $ \Lambda $; for instance, $ X_h $ and $ X_e $ are identical in $ \Lambda \subset \t{ob} (\mcal C_{NN}) $ when $ h\in H $.

We will recall the definition of $ \mcal A $ here.
For $ g_1,g_2 \in G $, we have a vector space $ \mcal A_{g_1,g_2}$ which is the quotient of the vector space $\us {W\in \t {ob} (\mcal C_{NN})} \bigoplus  \mcal C_{NN} \left(X_{g_1} \us N \otimes W \; , \; W \us N \otimes X_{g_2} \right)$ over the subspace generated by elements of the form $ \left[a \circ (\t {id}_{X_{g_1}} \us N \otimes f ) -  (f \us N \otimes \t {id}_{X_{g_2}}  ) \circ a \right]$ for $ a \in \mcal C_{NN} \left(X_{g_1} \us N \otimes Z \; , \; W \us N \otimes X_{g_2} \right)$ and $ f \in \mcal C_{NN} (W,Z) $.
We denote the quotient map by $ \psi_{g_1 , g_2} $. We will also use the notation $ \psi^W_{g_1,g_2} $ (resp., $ \psi^s_{g_1,g_2} $) for the restriction map $ \left. \us {}{\psi_{g_1,g_2}} \right|_{\mcal C_{NN} \left(X_{g_1} \us N \otimes W \; , \; W \us N \otimes X_{g_2} \right)} $ (resp., $ \left. \us {} {\psi_{g_1,g_2}} \right|_{\mcal C_{NN} \left(X_{g_1} \us N \otimes X_s \; , \; X_s \us N \otimes X_{g_2} \right)} $ for $ s\in G $).
Further, $ \mcal A^W_{g_1,g_2} $ and $ \mcal A^s_{g_1,g_2} $ will denote the range of the maps $ \psi^W_{g_1,g_2} $ and $ \psi^s_{g_1,g_2} $ respectively.

\noindent\textbf{Notation.} For any two vectors $ v_1 $ and $ v_2 $ in any vector space, we will write $ v_1 \sim v_2 $ when $ \t{span } v_1 = \t{span } v_2 $.
\begin{prop}
$ \mcal A_{g_1,g_2} $ is linearly spanned by the set
$
\left\{
\left. \psi^s_{g_1,g_2} \left( \; \raisebox{-4.4em}{\hexagon{$ g_1 $}{$ s $}{$e $}{$ s \:$}{$ g_2 $}{$ h_2 $}{$ h_1 $}{$e  $}{\psfrag{9}{$ g_1 s $}\psfrag{0}{$ s h_2 g_2 $}}} \hspace{.5em} \right) \right|
h_1,h_2 \in H,s\in G \right\}
$.
\end{prop}
\noindent We denote the above element by $ a(h_1 , g_1 , s , h_2 , g_2) $.
Note that $ h_1 g_1  = s \; h_2 g_2 \; s^{-1} $.
\begin{proof}
Since the weight set $\Lambda = \{X_s : s \in G\} $ is full, we may use the relation satisfied by the quotient map $ \psi_{g_1,g_2} $ to say  $ \mcal A_{g_1,g_2} = \t{span} \left(\us {s \in G} \cup \mcal A^s_{g_1,g_2}\right)$.
So, by Remark \ref{hexagon}, $ \mcal A_{g_1 , g_2} $ can be linearly generated by elements of the form $ \psi^s_{g_1,g_2} \left( \; \raisebox{-4.4em}{\hexagon{$ g_1 $}{$ s $}{$h_3 $}{$ s \:$}{$ g_2 $}{$ h_4 $}{$ h_1 $}{$h_2  $}{\psfrag{9}{$ g_1 h_3 s $}\psfrag{0}{$ s h_4 g_2 $}}} \hspace{.5em} \right) $ for $ h_1,h_2,h_3,h_4 \in H, s\in G $.

Using Proposition \ref{smbox*alg}, we may write \hspace{1em} \raisebox{-2.7 em}
{\psfrag{1}{\reflectbox{$h^{-1}_1 $}}
	\psfrag{2}{$ h_1 g h_2 $}
	\psfrag{3}{$ g $}
	\psfrag{4}{$ h_2 $}
	\psfrag{5}{\reflectbox{$h_1 $}}
	\psfrag{6}{$ g $}
	\psfrag{7}{$ h^{-1}_2 $}
	\includegraphics[scale=0.2]{figures/cnn/2disc.eps}} \hspace{1.5em}
$\sim$ \raisebox{-1.5 em}{\1disc{$e $}{$ g $}{$ g $}{$ e $}} 
$ = \t{id}_{X_g}$.
Again, using the relation satisfied by the quotient map and setting $ t:= h_3 s h^{-1}_2 $, we get
\[
\psi^s_{g_1,g_2} \left( \; \raisebox{-4.4em}{\hexagon{$ g_1 $}{$ s $}{$h_3 $}{$ s \:$}{$ g_2 $}{$ h_4 $}{$ h_1 $}{$h_2  $}{\psfrag{9}{$ g_1 h_3 s $}\psfrag{0}{$ s h_4 g_2 $}}} \hspace{.5em} \right) 
\; \sim \;
\psi^t_{g_1,g_2} \left( \left( \raisebox{-1.5 em}{\1disc{$h_3 $}{$ s $}{$ t $}{$ h_2 $}}
\hspace{1em} \us N \otimes \t{id}_{X_{g_2}} \right) \circ \;
\raisebox{-4.4em}{\hexagon{$ g_1 $}{$ s $}{$h_3 $}{$ s\: $}{$ g_2 $}{$ h_4 $}{$ h_1 $}{$h_2  $}{\psfrag{9}{$ g_1 h_3 s $}\psfrag{0}{$ s h_4 g_2 $}}} \hspace{1em}
\circ \left(\t{id}_{X_{g_1}} \us N \otimes \hspace{1em} \raisebox{-1.5 em}{\1disc{$h^{-1}_3 $}{$ t $}{$ s $}{$ h^{-1}_2 $}} 
\right) \right).
\]
We then apply Lemma \ref{tribox} (i) and (ii) to get
\[
\psi^s_{g_1,g_2} \left( \; \raisebox{-4.4em}{\hexagon{$ g_1 $}{$ s $}{$h_3 $}{$ s \:$}{$ g_2 $}{$ h_4 $}{$ h_1 $}{$h_2  $}{\psfrag{9}{$ g_1 h_3 s $}\psfrag{0}{$ s h_4 g_2 $}}} \hspace{.5em} \right)
\; \sim \; \psi^t_{g_1,g_2} \left(
\raisebox{-1.5em}{\wine{$ t $}{$ g_2 $}{$ th_2 h_4g_2 $}{\!\!$ h_2h_4 $}}
\; \circ  \hspace{4em}
\raisebox{-3.85 em}
{\psfrag{1}{\reflectbox{$h_3 $}}
	\psfrag{2}{$ h^{-1}_3 t h_2 h_4 g_2$}
	\psfrag{3}{$ t h_2 h_4 g_2 $}
	\psfrag{4}{$ e $}
	\psfrag{5}{\reflectbox{$h_1 $}}
	\psfrag{6}{$ g_1 t h_2 $}
	\psfrag{7}{$ h_2 $}
	\psfrag{8}{\reflectbox{$ e\: $}}
	\psfrag{9}{$ h^{-1}_2 $}
	\psfrag{0}{$ g_1 t$}
	\includegraphics[scale=0.2]{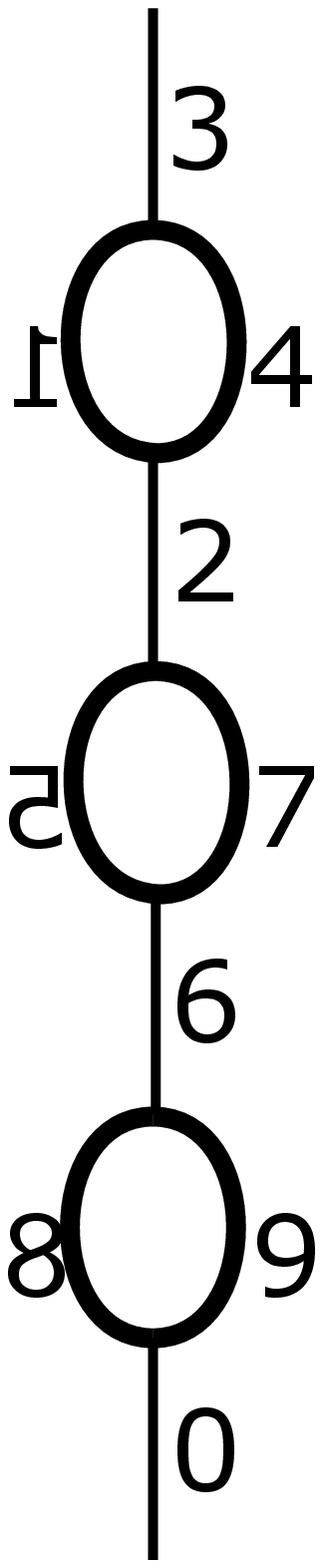}} \hspace{4em}
\circ \; 
\raisebox{-1.5em}{\tower{$ g_1 $}{$ t $}{$ g_1t $}{$ e $}}
\right).
\]
where the three vertically stacked discs correspond to their composition.
Once we apply the multiplication of these discs as stated in Proposition \ref{smbox*alg} (i), it becomes clear that the resultant (up to a unit scalar) is indeed of the form mentioned in the statement of this proposition.
\end{proof}

We will next unravel the multiplication in $ \mcal A $.
\begin{rem}\label{multrem}
Multiplication of affine annular morphisms is given by
\begin{align*}
\psi^t_{g_2 ,g_3} (c) \circ \psi^s_{g_1 ,g_2} (d)\;  = \; \psi^{X_{s} \us N \otimes X_{t}}_{g_1 , g_3} \left( \left(\t{id}_{X_s} \us N \otimes c\right) \; \circ \; \left(d \us N  \otimes \t{id}_{X_t}\right)  \right)
\end{align*}
for $ c \in \mcal C_{NN} (X_{g_2} \otimes X_t \; , \; X_t \otimes X_{g_3}) $ and $ d \in \mcal C_{NN} (X_{g_1} \otimes X_s \; , \; X_s \otimes X_{g_2}) $.
Using Proposition \ref{resid}, we can rewrite the above as
\begin{align*}
&\;\psi^t_{g_2 ,g_3} (c) \circ \psi^s_{g_1 ,g_2} (d)\\
=& \; \us {h \in H} \sum \psi^{sht}_{g_1 , g_3}
\left( \raisebox{-1.5em}{\tower{$ s $}{$ t $}{$ sht $}{$ h $}}
\us N \otimes  \t{id}_{X_{g_3}}  \; \circ \; \left(\t{id}_{X_s} \us N \otimes c\right) \; \circ \; \left(d \us N  \otimes \t{id}_{X_t}\right) \; \circ \; \t{id}_{X_{g_1}} \us N \otimes \; 
\raisebox{-1.5em}{\wine{$ s $}{$ t $}{$ sht $}{$ h $}}
\right).
\end{align*}
\end{rem}
\begin{prop}\label{mult}
\begin{align*}
& \left[\omega (h''_2, g_2, t)  \omega (t,h_3,g_3) \;  a(h''_2 , g_2 , t , h_3 , g_3)\right]
\; \circ \;
\left[ \omega (h_1 , g_1 , s) \omega (s, h'_2, g_2) \; a(h_1 , g_1 , s , h'_2 , g_2)\right]\\
= & \; \delta_{h'_2 = h''_2} \; \left[ \omega (s,t,h_3g_3) \ol \omega (s, h'_2 g_2, t)   \omega (h_1g_1 , s, t)    \right] \; \left[\omega (h_1 , g_1 , s t)\; \omega (st,h_3,g_3) \; a(h_1 , g_1 , st , h_3 , g_3)\right]
\end{align*}
\end{prop}
\begin{proof}
The above remark lets us express the element $ \left[a(h''_2 , g_2 , t , h_3 , g_3) \; \circ \; a(h_1 , g_1 , s , h'_2 , g_2)\right] $ as a sum over $ h\in H $ of
\begin{equation*}
\psi^{sht}_{g_1 , g_3}
\left( \us {\displaystyle =: b_h \t{ say}} {\underbracket{
\raisebox{-1.5em}{\tower{$ s $}{$ t $}{$ sht $}{$ h $}}
\us N \otimes  \t{id}_{X_{g_3}}  \; \circ
\; \left(  \t{id}_{X_s}  \us N \otimes \hspace{1em} 
\raisebox{-4.4em}{\hexagon{$ g_2 $}{$ t $}{$e $}{$ t\: $}{$ g_3 $}{$ h_3 $}{$ h''_2 $}{$e  $}{\psfrag{9}{$ g_2 t$}\psfrag{0}{$ t h_3 g_3 $}}}
\; \right) \; \circ \; \left( \;
\raisebox{-4.4em}{\hexagon{$ g_1 $}{$ s $}{$e $}{$ s \:$}{$ g_2 $}{$ h'_2 $}{$ h_1 $}{$ e $}{\psfrag{9}{$ g_1 s $}\psfrag{0}{$ s h'_2 g_2 $}}}
\hspace{1em} \us N \otimes \t{id}_{X_t} \right) \; \circ \; \t{id}_{X_{g_1}} \us N \otimes
\raisebox{-1.5em}{\wine{$ s $}{$ t $}{$ sht $}{$ h $}}
}}
\right).
\end{equation*}
We could use Lemma \ref{2tri} at three instances in the above expression of $ b_h $, and thereby we may rewrite $ b_h $ up to a unit scalar as
\begin{align}\label{step1}
	\hspace{1em}
\raisebox{-3em}{
	\psfrag{1}{\reflectbox{$ sht $}}
	\psfrag{2}{$ g_3 $}
	\psfrag{3}{\hspace{-3.5em}$shth_3 g_3 $}
	\psfrag{4}{$ h_3 $}
	\psfrag{5}{\reflectbox{$ s $}}
	\psfrag{6}{$ t h_3 g_3 $}
	\psfrag{7}{$ h $}
	\includegraphics[scale=0.2]{figures/cnn/winetower.eps}}
\hspace{2em} \circ \left(\t{id}_{X_s} \us N \otimes \;
\raisebox{-1.5 em}{\1disc{$h''_2 $}{$ g_2 t$}{$ t h_3 g_3 $}{$ e $}}\;\;
\right) \circ \hspace{1.5em}
\raisebox{-3em}{
	\psfrag{1}{\reflectbox{$ s $}}
	\psfrag{2}{$ g_2 t $}
	\psfrag{3}{\hspace{-3em}$s h'_2 g_2 t $}
	\psfrag{4}{$ h'_2 $}
	\psfrag{5}{\reflectbox{$ s h'_2 g_2 $}}
	\psfrag{6}{$ t  $}
	\psfrag{7}{$ e $}
	\includegraphics[scale=0.2]{figures/cnn/winetower.eps}}\hspace{1em}
\circ \left(
\raisebox{-1.5 em}{\1disc{$h_1 $}{$ g_1 s$}{$s h'_2 g_2 $}{$ e $}} \;
\us N \otimes \t{id}_{X_t} \right)\circ
\raisebox{-3em}{
	\psfrag{1}{\reflectbox{$ g_1s $}}
	\psfrag{2}{$ t $}
	\psfrag{3}{$g_1 sht $}
	\psfrag{4}{$ h $}
	\psfrag{5}{\reflectbox{$ g_1 $}}
	\psfrag{6}{$ sht $}
	\psfrag{7}{$ e $}
	\includegraphics[scale=0.2]{figures/cnn/winetower.eps}}
\comments{
\\
&\triup{sht}{g_3}{shth_3 g_3}{3}{3}{1} \; \circ \;
\left(  
 \ous
{\! \! \! \!   \displaystyle \tridown{s}{t h_3 g_3}{\! \! \! \! \! \! shth_3 g_3}{3}{3}{-0.5} }
{\t{id}_{X_s}  \us N \otimes \displaystyle \smbox{h''_2}{g_2 t}{t h_3 g_3}{e}{4}{2}{-0.3}}
{\! \! \! \! \! \! \! \! \! \! \! \!  \displaystyle \triup{s}{g_2 t}{\! \! \! sh'_2 g_2 t}{3}{3}{1}} 
\right) \; \circ \; \left( \ous
{\; \; \; \; \; \; \; \; \;  \displaystyle \tridown{s h'_2 g_2}{t}{\; \; \; \; \; \; \; sh'_2 g_2 t}{3}{3}{-0.5}}
{\displaystyle \smbox{h_1}{g_1 s}{sh'_2 g_2}{e \;\;\;}{4}{2}{-0.3} \us N \otimes \t{id}_{X_t}}
{\; \; \; \; \; \; \; \; \; \; \; \; \;  \displaystyle \triup{g_1 s}{t}{g_1 s h t}{3}{3}{1}} \right) \; \circ \; \tridown{g_1}{sht}{g_1 s h t}{3}{3}{-0.5}
}
\hspace{1.5em}.
\end{align}

In the above expression \ref{step1}, using Lemma \ref{tribox} (ii), we could make the disc in the fourth term pass through the bottom box in the third term to its top.
As a result, expression \ref{step1} turns out to be a scalar multiple of
\begin{align}\label{step2}
	\hspace{1em}
	\raisebox{-3em}{
		\psfrag{1}{\reflectbox{$ sht $}}
		\psfrag{2}{$ g_3 $}
		\psfrag{3}{\hspace{-3.5em}$shth_3 g_3 $}
		\psfrag{4}{$ h_3 $}
		\psfrag{5}{\reflectbox{$ s $}}
		\psfrag{6}{$ t h_3 g_3 $}
		\psfrag{7}{$ h $}
		\includegraphics[scale=0.2]{figures/cnn/winetower.eps}}
	\hspace{2em} \circ \left(\t{id}_{X_s} \us N \otimes \;
	\raisebox{-1.5 em}{\1disc{$h''_2 $}{$ g_2 t$}{$ t h_3 g_3 $}{$ e $}}\;\;
	\right) \circ\hspace{2em}
\raisebox{-4.4em}{\hexagon{$ g_1s $}{$ t $}{$e $}{$ s $}{$ g_2t $}{$ h'_2 $}{$ h_1 $}{$e  $}{\psfrag{9}{\hspace{-2em}$ g_1 st $}\psfrag{0}{\hspace{-3em}$ sh'_2g_2t $}}}
\hspace{1em}	\circ \;
	\raisebox{-3em}{
		\psfrag{1}{\reflectbox{$ g_1s $}}
		\psfrag{2}{$ t $}
		\psfrag{3}{$g_1 sht $}
		\psfrag{4}{$ h $}
		\psfrag{5}{\reflectbox{$ g_1 $}}
		\psfrag{6}{$ sht $}
		\psfrag{7}{$ e $}
		\includegraphics[scale=0.2]{figures/cnn/winetower.eps}}
	\hspace{1.5em}.
\end{align}
Observe that the composition of the bottom box of the third term and the top box in the fourth term (in expression \ref{step2}) is $ \t{id}_{X_{g_1 st}} $ if $ h=e $ and zero otherwise; this follows from Proposition \ref{resid}.
Similarly, Lemma \ref{tribox} (i) allows us to move the disc in the second term up through the bottom box of the first term, and thereby the expression \ref{step2} becomes a scalar multiple of
\begin{align}\label{step3}
	\hspace{1em}
	\raisebox{-3em}{
		\psfrag{1}{\reflectbox{$ sht $}}
		\psfrag{2}{$ g_3 $}
		\psfrag{3}{\hspace{-8.75em}($s h h''_2g_2 t = $)$shth_3 g_3 $}
		\psfrag{4}{$ h_3 $}
		\psfrag{5}{\reflectbox{$ s $}}
		\psfrag{6}{$ g_2 t $}
		\psfrag{7}{\!$ h h''_2 $}
		\includegraphics[scale=0.2]{figures/cnn/winetower.eps}}
	\hspace{1em} \circ\hspace{2em}
	\raisebox{-4.4em}{\hexagon{$ g_1s $}{$ t $}{$e $}{$ s $}{$ g_2t $}{$ h'_2 $}{$ h_1 $}{$e  $}{\psfrag{9}{\hspace{-2em}$ g_1 st $}\psfrag{0}{\hspace{-3em}$ sh'_2g_2t $}}}
	\hspace{1em}	\circ \;
	\raisebox{-3em}{
		\psfrag{1}{\reflectbox{$ g_1s $}}
		\psfrag{2}{$ t $}
		\psfrag{3}{$g_1 sht $}
		\psfrag{4}{$ h $}
		\psfrag{5}{\reflectbox{$ g_1 $}}
		\psfrag{6}{$ sht $}
		\psfrag{7}{$ e $}
		\includegraphics[scale=0.2]{figures/cnn/winetower.eps}}
	\hspace{1.5em}.
\end{align}
Again, by Proposition \ref{resid}, the composition of the bottom box of the first term and the top of the second term in expression \ref{step3} is $ \t {id}_{X_{shth_3 g_3}} $ if $ h h''_2 = h'_2 $ and zero otherwise.
\comments{
the third term in the above expression turns out to be a unit scalar times 
$ \ous
{\displaystyle \smbox{h_1}{g_1st}{sh'_2 g_2 t}{e}{4}{2}{-0.3}}
{\displaystyle \tridown{g_1s}{t}{}{3}{3}{-0.5}}
{\displaystyle \triup{g_1s}{t}{g_1s ht}{3}{3}{1}}
\t{ which is zero by Proposition \ref{resid} unless } h = e $.
So, we will assume $ h=e $.
Treating the second term in the same way, we will get that it is zero unless $ h'_2 = h''_2 $.
}

We now consider the case $ h=e $ and $ h'_2 = h''_2$ ($ = h_2$ say).
The above discussion implies that in this case, $ \left[a(h_2 , g_2 , t , h_3 , g_3) \; \circ \; a(h_1 , g_1 , s , h_2 , g_2)\right] $ is indeed a scalar multiple of $ a(h_1 , g_1 , st , h_3 , g_3) $.
So, we need to gather the $ 3 $-cocycle arising at various steps.
To obtain step \ref{step1}, Lemma \ref{2tri} will give the following six scalars
\[
\left[ \omega (s,t,h_3) \omega (s,th_3 , g_3)\right]\;
\left[\cancel{\ol \omega (sh_2,g_2,e)} \ol \omega (sh_2,g_2,t) \right]\;
\left[ \cancel{\omega (g_1 , s ,e)} \omega (g_1 , s ,t) \right].
\]
Application of Lemma \ref{tribox} (ii) (resp., (i)) while obtaining step \ref{step2} (resp., \ref{step3}) from step \ref{step1} (resp., \ref{step2}), yield
\[
\left[ \omega(h_1 , g_1 s, t) \cancel{\ol \omega (sh_2 g_2 ,e ,e)} \cancel{\ol \omega (h_1, g_1 s ,e)} \right] \;
\t{ (resp., } \left[ \cancel{\omega (s,th_3g_3,e)} \ol \omega (s, h_2, g_2 t) \cancel{\omega (s, e, h_2)} \right]
\t{ ).}
\]
\comments{
Applying Lemma \ref{tribox} (i) (resp., (ii)) and isometry property of upward pointing triangles, the second (resp., third) term of step \ref{step1} becomes 
\[
\left[ \cancel{\omega (s,th_3g_3,e)} \ol \omega (s, h_2, g_2 t) \cancel{\omega (s, e, h_2)} \right] \; \smbox{e}{sh_2g_2t}{sth_3g_3}{e}{4}{2}{-0.3}
\]
\[
\t{(resp., } \left[ \omega(h_1 , g_1 s, t) \cancel{\ol \omega (sh_2 g_2 ,e ,e)} \cancel{\ol \omega (h_1, g_1 s ,e)} \right]  \; \smbox{h_1}{g_1 s t}{sh_2g_2t}{e}{4}{2}{-0.3} \t{ ).}
\]
The product of the above two boxes is $\; \smbox{h_1}{g_1 s t}{sth_3g_3}{e}{4}{2}{-0.3} $.
}
Thus, we obtained the equation
\begin{align*}
& \; a(h_2 , g_2 , t , h_3 , g_3) \; \circ \; a(h_1 , g_1 , s , h_2 , g_2)\\
= & \; \left[ \omega (s,t,h_3) \omega (s,th_3 , g_3) \ol \omega (sh_2,g_2,t) \omega (g_1 , s ,t) \omega(h_1 , g_1 s, t)  \ol \omega (s, h_2, g_2 t)   \right]
a(h_1 , g_1 , st , h_3 , g_3).
\end{align*}
We will be done with the proof once we match the scalars.
Applying the $ 3 $-cocycle relation \ref{3coc} on first and second, third and sixth, fourth and fifth terms separately, we get
\begin{align*}
&\left[ \omega (st,h_3,g_3) \omega (s,t,h_3g_3) \ol \omega (t,h_3,g_3) \right] \; 
\left[ \ol \omega (s, h_2, g_2) \ol \omega (s, h_2 g_2, t) \ol \omega (h_2, g_2, t) \right]\\
&\left[ \omega (h_1g_1 , s, t) \omega (h_1 , g_1 , s t) \ol  \omega (h_1 , g_1 , s)  \right]
\end{align*}
\end{proof}
\noindent \textbf{Notation.} We see that $ \left[ \omega (h_1 , g_1 , s) \omega (s, h_2, g_2) \; a(h_1 , g_1 , s , h_2 , g_2)\right] $ is better behaved with respect to multiplication than $ a(h_1 , g_1 , s , h_2 , g_2)$.
So, we set\\
$ A(h_1 , g_1 , s , h_2 , g_2) := \left[ \omega (h_1 , g_1 , s) \omega (s, h_2, g_2) \; a(h_1 , g_1 , s , h_2 , g_2)\right] $ and the above proposition translates as:\\
$A(h''_2 , g_2 , t , h_3 , g_3) \; \circ \; A(h_1 , g_1 , s , h'_2 , g_2) \; =  \; \delta_{h'_2 = h''_2} \; \left[ \omega (s,t,h_3g_3) \ol \omega (s, h'_2 g_2, t)   \omega (h_1g_1 , s, t)    \right] \; A(h_1 , g_1 , st , h_3 , g_3)$.
\vskip 2 em
Next we will compute the canonical trace $ \Omega $ on $ \mcal A_{g,g} $ for $ g\in G $.
For this, we need orthonormal basis of $ \mcal C_{NN}  ( {}_N L^2 (N) {}_N , X_s) $ for $ s \in G $ with respect to the inner product given by
\[
 \mcal C_{NN}  ( {}_N L^2 (N) {}_N , X_s)  \times  \mcal C_{NN}  ( {}_N L^2 (N) {}_N , X_s)  \ni (c,d) \longmapsto d^* \circ c \in \C.
\]
By Proposition \ref{smbox}, $ \mcal C_{NN}  ( {}_N L^2 (N) {}_N , X_s) $ is zero unless $ s \in H $.
Now, $ X_e = X_h $ for all $ h\in H $.
Since $ N \subset Q $ is irreducible, the space $  \mcal C_{NN}  ( {}_N L^2 (N) {}_N , X_h)  $  is one-dimensional and spanned by the element the inclusion map $ \hat 1 \os {\displaystyle {\iota_h}} \longmapsto [1]_h $.
$ \iota^*_h $ is simply the conditional expectation $ E_N $.

The definition of $ \Omega $ then turns out to be (following \cite{GJ})
\[
\mcal A_{g,g} \ni \psi^s_{g,g} (c ) \os {\displaystyle \Omega} \longmapsto \us {s\in H} \sum \; R^*_g \circ \left(\t{id}_{X_{g^{-1}}} \us N \otimes \; \iota^*_s \; \us N \otimes \t{id}_{X_g}\right) \circ \left(\t{id}_{X_{g^{-1}}} \us N \otimes c\right) \circ \left(\t{id}_{X_{g^{-1}} \us N \otimes X_g} \; \us N \otimes  \iota_s\right) \circ R_g \; \in \C.
\]
\begin{prop}
$ \Omega \left(A(h_1, g,  s , h_2 , g)\right)  = \delta_{h_1 = h_2} \; \delta_{s=e} $.
\end{prop}
\begin{proof}
For $ h\in H $, we need to compute the scalar
\begin{align*}
& \; R^*_g \circ \left(\t{id}_{X_{g^{-1}}} \us N \otimes \; \iota^*_h \; \us N \otimes \t{id}_{X_g}\right) \circ \left(\t{id}_{X_{g^{-1}}} \us N \otimes \hspace{1em}
\raisebox{-4.4em}{\hexagon{$ g $}{$ h $}{$e $}{$ h $}{$ g $}{$ h_2 $}{$ h_1 $}{$ e $}{\psfrag{9}{$ g h $}\psfrag{0}{$ hh_2g $}}}
\; \right) \circ \left(\t{id}_{X_{g^{-1}} \us N \otimes X_g} \; \us N \otimes  \iota_h\right) \circ R_g \; (\hat 1)
\end{align*}
\begin{align*}
= & \; \us i \sum R^*_g \circ \left(\t{id}_{X_{g^{-1}}} \us N \otimes \; \iota^*_h \; \us N \otimes \t{id}_{X_g}\right) \circ \left(\t{id}_{X_{g^{-1}}} \us N \otimes
 \hspace{1em}
 \raisebox{-4.4em}{\hexagon{$ g $}{$ h $}{$e $}{$ h $}{$ g $}{$ h_2 $}{$ h_1 $}{$ e $}{\psfrag{9}{$ g h $}\psfrag{0}{$ hh_2g $}}}
\;\right) \left[ [u^* (g^{-1} , g) \alpha_{g^{-1}} (b_i)]_{g^{-1}} \us N \otimes [b^*_i]_g \us N \otimes [1]_h \right]\\
= & \; \abs {H}^{-1} \us {i,j} \sum R^*_g \circ \left(\t{id}_{X_{g^{-1}}} \us N \otimes \; \iota^*_h \; \us N \otimes \t{id}_{X_g}\right)\\
& \;\;\;\;\;\;\;\;\;\;\;\;\;\;\;\;  \left[ [u^* (g^{-1} , g) \alpha_{g^{-1}} (b_i)]_{g^{-1}} \us N \otimes [\alpha_{h_1} \left(b^*_i u(g,h) \right) \; u(h_1 , gh) \; u^*(h h_2 ,g) \; \alpha_h(b_j) ]_h \us N \otimes [\alpha_{h^{-1}_2} (b^*_j)]_g \right]\\
= & \; \abs {H}^{-1} \us {j} \sum R^*_g \left[ \left[u^* (g^{-1} , g) \alpha_{g^{-1}} \left( u(g,h) \; \alpha_{h^{-1}_1} \left(  u(h_1 , gh) \; u^*(h h_2 ,g) \alpha_{h} (b_j) \right) \right) \right]_{g^{-1}} \us N \otimes [\alpha_{h^{-1}_2} (b^*_j)]_g \right]\\
= & \; \abs {H}^{-1} \us {j} \sum u^* (g^{-1} , g) \alpha_{g^{-1}} \left( u(g,h) \; \alpha_{h^{-1}_1} \left(  u(h_1 , gh) \; u^*(h h_2 ,g) \alpha_{h} (b_j) \right) \alpha_{h^{-1}_2} (b^*_j) \right) u (g^{-1} , g)\\
= & \; \abs {H}^{-1} \us {j} \sum u^* (g^{-1} , g) \alpha_{g^{-1}} \left( u(g,h) \; \alpha_{h^{-1}_1} \left(  u(h_1 , gh) \; u^*(h h_2 ,g)  \right)  \right)\; \alpha_{g^{-1}} \left( \alpha_{h^{-1}_1 h } (b_j) \alpha_{h^{-1}_2} (b^*_j) \right) u (g^{-1} , g).
\end{align*}
Pulling the sum over the last term, we get $ \alpha_{g^{-1}} \left( \alpha_{h^{-1}_2} \left( \us {j} \sum \alpha_{h_2 h^{-1}_1 h} (b_j) b^*_j \right) \right) = \delta_{h=h_1 h^{-1}_2} \abs H$ (which is a standard fact in fixed-point subfactor of an outer action of finite group).
Let us assume $ h=h_1 h^{-1}_2 $.
But then, $ h_1 g h = h h_2 g$ will imply $ h_1 = h_2 $ and thereby $ h=e $.

Under the assumption $ h=e $ and $ h_1 = h_2 $, in the above expression, the term in between $ u^* (g^{-1} , g) $ and $ u (g^{-1} , g) $, becomes $ 1 $.
This gives the required result.
\end{proof}
\begin{cor}
The set $ \left\{ A(h_1, g_1, s, h_2,g_2) :h_1,h_2 \in H, s\in G \t{ such that } h_1g_1s =sh_2g_2 \right\} $ is a basis for $ \mcal A_{g_1,g_2} $.
\end{cor}
\begin{proof}
This easily follows from that $ \Omega  $ is non-degenerate on $ \mcal A $ (which is a consequence of $ \Omega $ being positive (see \cite{GJ})).
\end{proof}
\vskip 2em
We will now describe the $ * $-structure on $ \mcal A $ which we denote by $ \# $.
From \cite{GJ}, the definition of $ (\psi^s_{g_1 ,g_2} (c))^\# $ is the following:
\[
\psi^{s^{-1}}_{g_2 ,g_1} \left( \left( \t{id}_{X_{s^{-1}}} \us N \otimes \t{id}_{X_{g_1}} \us N \otimes \ol R^*_s \right) \circ \left(\t{id}_{X_{s^{-1}}} \us N \otimes c^*  \us N \otimes   \t{id}_{X_{s^{-1}}}\right)  \circ \left( R_s  \us N \otimes \t{id}_{X_{g_2}} \us N \otimes \t{id}_{X_{s^{-1}}} \right)\right) \in \mcal A_{g_2,g_1}.
\]
\begin{prop}\label{hash}
$ \left(A(h_1,g_1,s,h_2,g_2) \right)^\#$\\
\[
= \; \ol \omega (h_1 g_1 , s , s^{-1}) \; \omega (s , h_2 g_2 , s^{-1}) \; \ol \omega (s , s^{-1} ,h_1 g_1) \; A(h_2, g_2, s^{-1}, h_1,g_1).
\]
\end{prop}
\begin{proof}
Set $ A'(h_1,g_1,s,h_2,g_2) :=  \ol \omega (h_1 g_1 , s , s^{-1}) \; \omega (s , h_2 g_2 , s^{-1}) \; \ol \omega (s , s^{-1} ,h_1 g_1) \; A(h_2, g_2, s^{-1}, h_1,g_1)$.
Now, we get an inner product $ \lab \cdot , \cdot \rab' $ defined as
\[
\left \lab \; A(h_1,g_1,s,h_2,g_2) \; , \; A(h_3,g_3,t,h_4,g_4)\;  \right \rab' := \; \Omega \left( \; A'(h_3,g_3,t,h_4,g_4) \; A(h_1,g_1,s,h_2,g_2) \right)
\]
and extended linearly in the first and conjugate-linearly in the second variable.
In fact, the basis elements are orthonormal with respect to $ \lab \cdot , \cdot \rab' $.
Since $ \Omega \circ \# = \ol \Omega $ (by positivity of $ \Omega  $ (\cite{GJ})), it will be enough to prove $ \left(A(h_1,g_1,s,h_2,g_2) \right)^\# \sim  A'(h_1,g_1,s,h_2,g_2) $.
This is equivalent to proving $ \left(a(h_1,g_1,s,h_2,g_2) \right)^\# \sim a(h_2, g_2, s^{-1}, h_1,g_1) $.
This will follow from
\[
\us {\displaystyle A} {\underbracket
{ \t{id}_{X_{s^{-1}}} \us N \otimes \t{id}_{X_{g_1}} \us N \otimes \ol R^*_s }}
\;\; \circ \;\;
\us {\displaystyle B} {\underbracket
{\t{id}_{X_{s^{-1}}} \us N \otimes \hspace{1em}
\raisebox{-4.4em}{\hexagon{$ s $}{$ g_2 $}{$h_2 $}{$ g_1 $}{$ s $}{$ e $}{$ h^{-1}_1 $}{$ e $}{\psfrag{9}{$ sh_2 g_2 $}\psfrag{0}{$ g_1s $}}}
\hspace{1em}\us N \otimes   \t{id}_{X_{s^{-1}}}}}
\;\; \circ \;\;
\us {\displaystyle C} {\underbracket
{R_s  \us N \otimes \t{id}_{X_{g_2}} \us N \otimes \t{id}_{X_{s^{-1}}} }}
\hspace{1em} \sim \hspace{1em}
\raisebox{-4.4em}{\hexagon{$ g_2 $}{$ s^{-1} $}{$ e $}{$ s^{-1} $}{$ g_1 $}{$ h_1 $}{$ h_2 $}{$ e $}{\psfrag{9}{$ g_2 s^{-1} $}\psfrag{0}{$ s^{-1} h_1 g_1 $}}}
\]
The right side acting on $ [x]_{g_2} \us N \otimes [y]_{s^{-1}} $ gives (up to a nonzero scalar)
\begin{equation}\label{hash1}
\us i \sum
\left[
\alpha_{h_2} \left( x \alpha_{g_2} (y) \; u (g_2 ,s^{-1}) \right)\;
u(h_2,g_2 s^{-1}) u^* (s^{-1} h_1 ,g_1) u^* (s^{-1},  h_1) \alpha_{s^{-1}} (b_i)
\right]_{s^{-1}} \us N \otimes [\alpha_{h^{-1}_1} (b^*_i)]_{g_1}
\end{equation}
Next we compute the left side acting on $ [x]_{g_2} \us N \otimes [y]_{s^{-1}} $ (up to a nonzero scalar) in the following way
\begin{align*}
\os {\displaystyle C} \longmapsto & \us i \sum [u^* (s^{-1} , s) \alpha_{s^{-1}} (b_i) ]_{s^{-1}} \us N \otimes [ b^*_i]_s \us N \otimes [x]_{g_2} \us N \otimes [y]_{s^{-1}}
\\
\os {\displaystyle B} \longmapsto & \us {i,j} \sum [u^* (s^{-1} , s) \alpha_{s^{-1}} (b_i) ]_{s^{-1}}\\
& \us N \otimes \left[\alpha_{h^{-1}_1} \left( b^*_i \alpha_s (\alpha_{h_2} (x)) \; u(s,h_2) u(sh_2 , g_2) \; \right) u(h^{-1}_1 , sh_2 g_2) u^* (g_1 ,s ) \alpha_{g_1} (b_j) \right]_{g_1} \us N \otimes [b^*_j]_s \us N \otimes [y]_{s^{-1}}
\\
\os {\displaystyle A} \longmapsto & \us {i,j} \sum [u^* (s^{-1} , s) \alpha_{s^{-1}} (b_i) ]_{s^{-1}} \us N \otimes\\
& \left[\alpha_{h^{-1}_1} \left( b^*_i \alpha_s (\alpha_{h_2} (x)) \; u(s,h_2) u(sh_2 , g_2) \; \right) u(h^{-1}_1 , s h_2 g_2) u^* (g_1 ,s ) \alpha_{g_1} (b_j) \right]_{g_1}
E_N \left( b^*_j \alpha_s (y) u(s,s^{-1}) \right)
\\
\sim \;\; &  \us {i} \sum [u^* (s^{-1} , s) \alpha_{s^{-1}} (b_i) ]_{s^{-1}}\\
& \us N \otimes \left[\alpha_{h^{-1}_1} \left( b^*_i \alpha_s (\alpha_{h_2} (x)) \; u(s,h_2) u(sh_2 , g_2) \; \right) u(h^{-1}_1, sh_2 g_2) u^* (g_1 ,s ) \alpha_{g_1} \left(\alpha_s (y) u(s,s^{-1}) \right) \right]_{g_1}\\
= \;\; &  \us {i,k} \sum [u^* (s^{-1} , s) \alpha_{s^{-1}} (b_i) ]_{s^{-1}} \us N \otimes \\
& \! \! \! \! \! \! \! \left[ E_N \left(b^*_i \alpha_s (\alpha_{h_2} (x)) \; u(s,h_2) u(sh_2 , g_2) \alpha_{h_1} \left(   u(h^{-1}_1, sh_2 g_2) u^* (g_1 ,s ) \alpha_{g_1} \left(\alpha_s (y) u(s,s^{-1}) \right) \right) b_k \right) \alpha_{h^{-1}_1} (b^*_k) \right]_{g_1}\\
=  \;\; &  \us {k} \sum \left[u^* (s^{-1} , s) \alpha_{s^{-1}} \left(\alpha_s (\alpha_{h_2} (x)) \; u(s,h_2) u(sh_2 , g_2) \alpha_{h_1} \left(   u(h^{-1}_1, sh_2 g_2) u^* (g_1 ,s ) \alpha_{g_1} \left(\alpha_s (y) u(s,s^{-1}) \right) \right) b_k \right) \right]_{s^{-1}}\\
& \; \us N \otimes \left[ \alpha_{h^{-1}_1} (b^*_k) \right]_{g_1}\\
\end{align*}
Since the second tensor component matches with that of the expression in \ref{hash1}, we will now work with the first term.
\begin{align*}
& \; u^* (s^{-1} , s) \alpha_{s^{-1}} \left(\alpha_s (\alpha_{h_2} (x)) \; u(s,h_2) u(sh_2 , g_2) \alpha_{h_1} \left(   u(h^{-1}_1, sh_2 g_2) u^* (g_1 ,s ) \alpha_{g_1} \left(\alpha_s (y) u(s,s^{-1}) \right) \right) b_k \right)\\
= & \; \alpha_{h_2} (x) \; u^* (s^{-1} , s) \; \alpha_{s^{-1}} \left(  u(s,h_2) u(sh_2 , g_2) \alpha_{h_1} \left(   u(h^{-1}_1, sh_2 g_2) u^* (g_1 ,s ) \alpha_{g_1} \left(\alpha_s (y) u(s,s^{-1}) \right) \right)  \right) \alpha_{s^{-1}} (b_k)\\
\end{align*}
In the last expression, we pick $ y $ and using the intertwining relation between $ u $ and $ \alpha $, we push it leftwards all the way to the right side of the term $ \alpha_{h_2} (x) $ and it becomes $ \alpha_{h_2} (\alpha_{g_2} (y)) $.
This matches the first two and the last terms with that of the first tensor component of the expression \ref{hash1}.
We are left with showing the $ u $-terms in the middle, namely
\begin{equation}\label{hash2}
u^* (s^{-1} , s) \; \alpha_{s^{-1}} \left(  u(s,h_2) u(sh_2 , g_2) \alpha_{h_1} \left(   u(h^{-1}_1, sh_2 g_2) u^* (g_1 ,s ) \alpha_{g_1} \left( u(s,s^{-1}) \right) \right)  \right)
\end{equation}
is a nonzero multiple of the $ u $-terms in \ref{hash1}, that is,
\begin{equation}\label{hash3}
\alpha_{h_2} \left(u (g_2 ,s^{-1}) \right)\;
u(h_2,g_2 s^{-1}) u^* (s^{-1} h_1 ,g_1) u^* (s^{-1},  h_1)
\end{equation}
Taking the adjoint of \ref{hash2} and \ref{hash3} separately, we get the same automorphism $ \alpha_{h_2} \alpha_{g_2} \alpha_{s^{-1}} \alpha^{-1}_{g_1} \alpha^{-1}_{h_1} \alpha^{-1}_{s^{-1}}$.
Hence, we are done.
\end{proof}
\vskip 2em
In order to describe the representations of $ \mcal A $, we need a few more notations.
As in Section \ref{diag}, $ \mscr C $ will denote the set of conjugacy classes, $ g_C $ will be a representative of $ C \in \mscr C $  and for $ g\in C $, we pick $ w_g $ such that $ g=w_g g_C w^{-1}_g $.
Also, $ \vphi_C $ will be the $ 2 $-cocycle $ \ol \vphi_{g_C} $ of $ G_C $.
For $ C \in \mscr C $, set $ S_C:= \{ (h,g) \in H\times G : hg \in C\} $.
\begin{thm}\label{bhtubalg}
(i) The affine annular algebra $ \mcal A = (( \mcal A_{g_1 , g_2} ))_{\t {fin. supp.}}$
is isomorphic as a $ * $-algebra to $\us {C \in \mscr C} \bigoplus M_{S_C} \otimes [\C G_C]_{\vphi_C}$ where $ M_{S_C} $ denotes the $ * $-algebra of finitely supported matrices with rows and columns indexed by elements of $ S_C $.\\
(ii) Every Hilbert space representation $ \pi : \mcal A \ra \mcal L (V) $ decomposes uniquely (up to isomorphism) as an orthogonal direct sum of submodules $ V^C := \left\lab \t{Range } \pi \left(a(e,g_C, e , e, g_C)\right) \right\rab$ for $ C \in \mscr C $.
(We will call a representation of $ \mcal A $ `supported on $ C \in \mscr C$' if it is generated by the range of the action of the projection $ a(e,g_C, e , e, g_C) $.)
The category of $ C $-supported representations of $ \mcal C $ is additively equivalent to representation category of $ [\C G_C]_{\vphi_C} $.

\end{thm}
\begin{proof}
(i) Define the map $ \Phi : \mcal A \lra  \us {C \in \mscr C} \bigoplus M_{S_C} \otimes [\C G_C]_{\vphi_C}  $ by
\[
a(h_1,g_1,s,h_2,g_2) \os {\displaystyle \Phi} \longmapsto \ol \gamma_{g_C , w_{h_1 g_1} , w_{h_2 g_2}} E_{(h_2,g_2) , (h_1,g_1)} \otimes [w^{-1}_{h_2 g_2} s^{-1} \; w_{h_1g_1} ]
\]
extended linearly where $ h_1g_1 , h_2 g_2 \in C $.
Using the formula for multiplication and $ \# $ in Propositions \ref{mult} and \ref{hash} and the cocyle relation in Proposition \ref{2coc}, one can imitate the proof of Proposition \ref{diagtubalg} to show that the map $ \Phi $ serves as the required isomorphism.

(ii) Let $\pi: \mcal A \ra \mcal L (V) $ be a Hilbert space representation.
For $ C_1,C_2 \in \mscr C $ such that $ C_1 \neq C_2 $, we need to show $ V^{C_1} $ and $ V^{C_2} $ are orthogonal.
Taking inner product of the generating vectors, we get $ \left\lab \pi (a(e,g_{C_1},s_1,h_1, g)) \xi\; ,\; \pi ( a(e,g_{C_2}, s_2,h_2,g)) \eta \right\rab = \left\lab \pi (a(e,g_{C_2}, s_2,h_2,g)) ^\# \cdot a(e,g_{C_1},s_1,h_1, g)) \xi\; , \; \eta \right\rab$ which is zero unless $ h_1 = h_2 $ but in that case $ C_1 $ and $ C_2 $ have to be the same; so, the inner product is zero.

For the decomposition, it remains to show that $ V \subset \us {C\in \mscr C} \bigoplus V^C $.
Let $ \xi \in V_g $.
Note that the identity $ \us {h \in H} \sum a(h,g,e,h,g) $ of $ \mcal A_{g,g} $ is a sum of orthogonal projections.
So, $ \xi = \us {h \in H} \sum \pi (a(h,g,e,h,g)) \xi  $.
For $ h\in H $, we have $ \pi (a(h,g,e,h,g)) \xi = \pi (a(e,g_{C},e,h,g) \zeta$ where $hg \in C \in \mscr C $ and $\zeta = \pi (a(h,g,e,e,g_C)) \xi \in V^C$.

The proof of equivalence of $ C $-supported representations with representations of $ [\C G]_{\vphi_C} $ is exactly the same as the proof of Theorem \ref{diagtubalg}.
\end{proof}
\begin{rem}
To find the tube algebra $ \mcal T $ of $ \mcal C_{NN} $, we need to first choose a set of representatives in the isomorphism classes of simple objects in $ \mcal C_{NN} $.
By Propositions \ref{smbox} and \ref{smbox*alg}, $ X_{g_1} $ and $ X_{g_2} $ are isomorphic if and only if $ g_1 $ and $ g_2 $ are in the same $ H $-$ H $ double coset where the isomorphism is implemented by \raisebox{-1.5 em}{\1disc{$h_1 $}{$ g_1 $}{$ g_2 $}{$ h_2 $}} for any $ h_1, h_2 \in H $ satisfying $h_1 g_1 = g_2 h_2$.
Now for $ g\in G $, the endomorphism space $ \t {End} (X_g) $ is isomorphic to the group algebra $ H^g := H \cap g^{-1} H g $ twisted by the scalar $ 2 $-cocycle $ H^g \times H^g  \ni (h_1 , h_2) \mapsto \overline{\omega} (g h_1 g^{-1}, g h_2 g^{-1}, g) \; \omega (g h_1 g^{-1} , g , h_2) \; \overline{\omega} (g , h_1, h_2) \in S^1$ via
\[
\C H^g \supset H^g \ni h \longmapsto \hspace{1.5em} \raisebox{-1.5 em}{\1disc{$ g h g^{-1} $}{$ g $}{$ g $}{$ h $}} \in \t {End} (X_g).
\]
For $ g \in G $, fix a maximal set $ \Pi_g $ of mutually orthogonal minimal projections in $ \t {End} (X_g) $.
Let $ H \backslash G / H $ be a set of representatives from all the $ H $-$ H $ double cosets in $ G $.
Then, it follows that
$ \us {g \in  H \backslash G / H }  \bigcup  \left\{  \t{Range} (p) : p \in  \Pi_g \right\}$ is a set of representatives in the isomorphism classes of simple objects in $ \mcal C_{NN} $.
Hence, the tube algebra $ \mcal T $ is isomorphic (as a $ * $-algebra) to
\[
\us {g_1 , g_2 \in H \backslash G / H}  \bigoplus \;\us {p_1 \in \Pi_{g_1}, p_2 \in \Pi_{g_2}}  \bigoplus \left[\psi^{\mathbbm 1}_{g_2,g_2} (p_2) \circ
{\mcal A}_{g_1,g_2} \circ \psi^{\mathbbm 1}_{g_1,g_1} (p_1)\right].
\]
\end{rem}
\begin{rem}
By \cite[Theorem 4.2]{GJ}, we know that the representation categories of the  affine algebra $ \mcal A $ and tube algebra $ \mcal T $ are equivalent although as $ * $-algebras they are non-isomorphic.
There is one thing to notice that this representation category (appearing in Theorem \ref{bhtubalg}) is also equivalent to the category of tube representations of the diagonal subfactor (as in Theorem \ref{diagtubalg}) corresponding to the automorphisms $\alpha_g$, $g\in H\cup K$ of the $ II_1 $-factor $ Q $.
This equivalence can be seen in an alternative way:\\
\indent Let $_A \mcal H_B $ be an extremal bifinite bimodule and $ P $ be its corresponding subfactor planar algebra (namely the `unimodular bimodule planar algebra', in the sense of \cite{pertpa}).
By Theorem 4.2 of \cite{GJ}, the category of (Jones) affine $ P $-modules is equivalent to the representation category of the tube algebra of $ \mcal C_{AA}:= $ the category of bifinite $ A $-$ A $-bimodules generated by $_A \mcal H_B $.
By \cite[Remark 2.16]{DGG1}, the affine module categories corresponding to $ P $ and its dual $ \ol P $ are equivalent.
On the other hand, the dual planar algebra is isomorphic to the subfactor planar algebra associated to the contragradient bimodule $_B \ol {\mcal H}_A $.
Thus the representation category of the tube algebra of $ \mcal C_{AA} $ is equivalent to that of $ \mcal C_{BB}:= $ the category of bifinite $ B $-$ B $-bimodules generated by $_A \mcal H_B $.\\
\indent Next, consider an intermediate extremal finite index subfactor $ N \subset Q \subset M $.
Let $ \Gamma $ denote the bifinite bimodule $_N L^2 (M)_Q  $.
It is easy to check that the category $ \mcal C_{NN} $ of bifinite $ N $-$ N $-bimodules generated by $ \Gamma $ is the same as those which come from the subfactor $ N \subset M $.
Let $ \mcal C_{QQ} $ denote the smallest $ C^* $-tensor category of bifinite $ Q $-$Q $-bimodules coming from the subfactor $ N \subset Q $ as well as $ Q \subset M $.
One can verify that $ \mcal C_{QQ} $ is the same as the category of bifinite $ Q $-$Q $-bimodules generated by $ \Gamma $.
Hence, from the previous paragraph, the category of tube representations of $ \mcal C_{NN} $ is equivalent to that of $ \mcal C_{QQ} $.\\
\indent Coming back to our context of Bisch-Haagerup subfactor $ N = Q^H \subset Q \rtimes K $ as set up in the beginning of this section, it remains to show that $ \mcal C_{QQ} $ is the $ C^* $-tensor category generated by the bimodules $_Q L^2 (Q_{\alpha_g})_Q $ for $ g \in H \cup K $; this is an easy computation.
\end{rem}
\bibliographystyle{alpha}

\end{document}